\documentclass{article}
\usepackage{graphicx} 
\usepackage{hyperref}
\usepackage{color}
\usepackage{mathrsfs}
\usepackage{eufrak}
\usepackage[mathscr]{eucal}
\usepackage{stmaryrd}
\usepackage{amsfonts}
\usepackage[english]{babel}
\usepackage{bm}
\usepackage{tikz-cd}
\usepackage{mathtools}
\usepackage{amsmath,amssymb}
\usepackage{amsthm}
\usepackage{geometry}
\usepackage[math]{cellspace}
\DeclarePairedDelimiter\inn{\langle}{\rangle}

\usepackage{mathtools, slashed}

\newcommand{\pd}[2]{\dfrac{\partial #1}{\partial #2}}

\newcommand{\pdo}[1]{\dfrac{\partial}{\partial #1}}

\usepackage{theoremref}
\usepackage{makeidx}
\usepackage{mathrsfs}

\newtheorem{proposition}{Proposition}[section]
\newtheorem{theorem}{Theorem}[section]

\DeclareMathOperator{\Sol}{Sol}

\newcommand{\RomanNum}[1]
    {\MakeUppercase{\romannumeral #1}}

\title{Ruled Minimal Surfaces in the 3-dimensional Solvable Group}
\author{Tiago Franco Ferreira\thanks{Universidade Estadual de Campinas-UNICAMP, email:t194445@dac.unicamp.br }}
\date{}

\begin{document}

\maketitle

\begin{abstract}
    We study ruled minimal surfaces in a family of three-dimensional solvable Lie groups depending on a real parameter $p$ and endowed with left-invariant metrics. For $p\neq 0$, we classify all orthogonal ruled surfaces and construct new examples of minimal surfaces. For $p=0$, we parametrize all ruled surfaces and obtain additional examples of minimal surfaces.
\end{abstract}

\vspace{0.1cm} 

\vspace{0.1cm}

\section{Introduction}

$\indent$Among the three-dimensional homogeneous spaces, Thurston's eight model geometries provide a particularly rich setting for studying these surfaces. The 3-dimensional solvable group, $\Sol(3)$, is one of these standard geometries. In this paper, we focus on finding family of examples of ruled minimal surfaces within a specific family of left-invariant metrics on $\Sol(3)$, motivated by non-standard algebraic structures.
\\

To fully appreciate the geometry of $\Sol(3)$, it is instructive to place it within the broader landscape of ruled minimal surfaces in three-dimensional homogeneous manifolds.The classification of these surfaces has been a prominent research direction since Catalan's classical theorem in Euclidean space (Section 3 of \cite{book 1}). This enduring interest is driven by their role as a rich source of both minimal and developable submanifolds, as exemplified in chapter 1, sections 1.2 and 1.3, of \cite{article 9}. In recent years, significant progress has been made across various spaces, systematically extending the theory beyond standard space forms. For instance, complete classifications of ruled minimal surfaces have been achieved in the three-dimensional Heisenberg group, which corresponds to the nilpotent geometry $Nil_3$ (Shin et al., 2013), and in the Berger spheres (Shin et al., 2023). Similar characterizations have been successfully established for product spaces such as $\mathbb{S}^2 \times \mathbb{R}$ and $\mathbb{H}^2 \times \mathbb{R}$, where nontrivial ruled minimal surfaces are strictly helicoids (Kim et al., 2009). Within this continuum of 3-manifolds, $\Sol(3)$ presents a distinctly different geometric behavior. Unlike the Heisenberg group, $\Sol(3)$ possesses a solvable but non-nilpotent algebraic structure, which profoundly impacts the behavior of its geodesics and the resulting ruled surfaces.
\\

The group $\Sol(3)$ is traditionally identified with $(\mathbb{R}^3\,,\,\cdot )$, where the group operation is defined by 
$$(x, y, z) \cdot (x', y', z') = \left(x+x' ,\ \exp(Ax)(y', z') + (x, y)\right) \quad \text{for } A \in \mathbb{M}_{2\times 2}(\mathbb{R}) .$$ 
\\[-7pt]

Consequently, the Lie algebra of this space can be written as a semi-direct product $\mathfrak{g}(A)= \mathbb{R}\times_{A}\mathbb{R}^{2}$ generated by $e_{1}=(1,0,0)$, $e_{2}=(0,1,0)$, and $e_{3}=(0,0,1)$. The standard commutation relations are given by $[e_{1}, e_{2}] = Ae_{2}$, $[e_{1}, e_{3}] = Ae_{3}$, and $[e_{2}, e_{3}] = 0$.
\\

According to a classical result (Theorem 1.4 of Chapter 7 from \cite{book 4}), every real three-dimensional Lie algebra that is solvable and non-nilpotent is isomorphic to a product $\mathbb{R}\times_{\theta}\mathbb{R}^{2}$
where 
$$\theta\in \left\{
\left(\begin{matrix}
	1&1\\
	0&1
\end{matrix}\right),\ 
\left(\begin{matrix}
	1&0\\
	0&\gamma
\end{matrix}\right) \text{ where } |\gamma|\le 1,\ 
\left(\begin{matrix}
	p&-1\\
	1&p
\end{matrix}\right) \text{ where } p\in \mathbb{R}
\right\} . $$
\\

While previous works, such as \cite{article 4}, typically investigate this algebra using the standard diagonal indexing matrix $A = \operatorname{diag}(1, -1)$, we direct our attention to the specific case where $$A = \begin{pmatrix} p & -1 \\ 1 & p \end{pmatrix}$$ for $p \in \mathbb{R}$. This choice of indexing matrix leads to the non-standard commutation relations $[e_{1}, e_{2}] = p e_{2} + e_{3}$, $[e_{1}, e_{3}] = -e_{2} + p e_{3}$, and $[e_{2}, e_{3}] = 0$. Equipping this space with the dual coframe $(\lambda^{1},\lambda^{2},\lambda^{3})$, we consider the family of left-invariant metrics $$g_{\varepsilon,\alpha} = \varepsilon^{2}(\lambda^{1})^{2} + \alpha^{2}((\lambda^{2})^{2} + (\lambda^{3})^{2})$$ for $\varepsilon, \alpha \in \mathbb{R}_{>0}$.
\\

In Section 1, we define ruled surfaces immersed in Riemannian manifolds and demonstrate that these surfaces are naturally characterized by a specific form of the manifold's geodesic equation. We then specialize our analysis to ambient manifolds that are Lie groups equipped with a left-invariant metric, showing that the tangent vectors of the immersed ruled surfaces can be fully determined using only information from the group's Lie algebra. In Section 2, we use the Koszul formula to compute the covariant derivative associated with the left-invariant metric $g_{\varepsilon, \alpha}$ and, in Theorem \ref{Th1 Sol3 gen}, we characterize all possible tangent frames for orthogonal ruled surfaces when $p \neq 0$. Based on this, Theorem \ref{Ruled Surf Sol(3) geral} provides explicit parameterizations of these orthogonal ruled surfaces, allowing us to identify families of examples of minimal surfaces in Theorem \ref{The Min Surf Geral}. Finally, in Section 3, we further explore the specific case where $p=0$ to obtain parameterizations for all possible ruled surfaces (Theorem \ref{Ruled Surf Param Espec}) and identify additional minimal surfaces in Theorem \ref{The Min Surf Espec}.

\section{Generalities on Ruled Surfaces}
\subsection{Ruled Surfaces in Riemannian Manifolds}
Given a 3-dimensional Riemannian manifold, $(M\,,\,g)$, a smooth curve $\gamma:(t_{1}\,,\,t_{2})\to M $ and a nowhere vanishing smooth vector field along $\gamma$, $X:(t_{1}\,,\,t_{2})\to TM$, a smooth ruled surface is given by
$$F(t\,,\,s) \ 
=\ \exp_{\gamma(t)}(s\,X(t)) . $$

Also $s$ is called the rulling direction, the map $s\mapsto F(t_{0}\,,\,s) $ is called the rulling or generator and $dF(\partial_{s})$ the rulling vector field.
From such definition it is noticeable that, being $\nabla$ the Levi-Civita connection, $F_{s}\ =\ dF(\partial_{s})$, then
\begin{equation}\label{Ruled Cond}
    \nabla_{F_{s}}F_{s}\ =\ 0 .
\end{equation}

Therefore, being $(W\,,\,x) $ a local chart of M, $F((t_{1},t_{2})\times (s_{1},s_{2}))\subset W$ and $(F^{1},F^{2},\dots, F^{n}) $ the component functions of F in such chart, in order to calculate the ruled minimal surfaces in M, we have to solve the system of second order differential equations:
$$\forall\,j=1\,,\,2\,,\,3\,,\ \ 
\pd{^{2}F^{j}}{s^{2}}(t\,,\,s)\ +\ \Gamma^{j}_{ik}(F(t\,,\,s))\pd{F^{i}}{s}(t\,,\,s)\pd{F^{i}}{s}(t\,,\,s)\ =\ 0. $$

\subsection{Ruled Surfaces in General 3-Dimensional Lie Groups}\label{Sec 3}
In this section we will discuss the case of maps $F:U\to G\ (U\subset \mathbb{R}^{2} \text{ open }) $, where G is a Lie Group endowed with a left invariant metric $g$. Let $\{u_{1}\,,\,u_{2}\,,\,u_{3}\}$ be a left invariant global basis of vector fields and $\{C_{ki}^{j}\}_{i\,,\,k\,,\,j=1}^{3} $ the structure constants of the Lie Algebra $\mathfrak{g}$ of G, that is,
$$\forall\,i\,,\,k=1\,,\,2\,,\,3\,,\ 
[u_{i}\,,\,u_{k}]\,=\, C_{ik}^{1}u_{1}\ +\ 
C_{ik}^{2}u_{2}\ +\ 
C_{ik}^{3}u_{3} .$$

Now, if $\{u_{1}\,,\,u_{2}\,,\,u_{3}\}$ is orthonormal, then the Levi-Civita connection associated to $g$ will be given by:
$$\nabla_{u_{k}}u_{i}\ =\ \sum_{j}\,L_{ik}^{j},u_{j} \ \ ,\ \ \ 
L_{ik}^{j}\ =\ \frac{1}{2}\,\left(C_{ki}^{j} \ -\ C_{ij}^{k} \ +\ C_{jk}^{i}\right) .$$

This means that, being $(W\,,\,x) $ a local chart of G with $F(U)\subset W $, if we write 
$$TF(\partial_{2})\,=\,F_{s}\ 
=\ F_{s}^{1}\,u_{1}\ +\ 
F_{s}^{2}\,u_{2}\ +\ 
F_{s}^{3}\,u_{3} \ 
=\ \pd{F^{1}}{s}\,\pdo{x^{1}} \ +\ 
\pd{F^{2}}{s}\,\pdo{x^{2}}\ +\ 
\pd{F^{3}}{s}\,\pdo{x^{3}} .$$
then, since, in the left invariant frame, the symbols $L_{ik}^{j}$ are simply real constants, we can separate the second order system of differential equations $\nabla_{F_{s}}F_{s}=0 $ into two systems of first order equations:
\begin{equation}\label{general V2 eq}
    \forall\,j=1\,,\,2\,,\,3\,,\ \ 
\pd{F_{s}^{j}}{s}\ +\ \sum_{k\,,\,i=1}^{3}L^{j}_{ik}F_{s}^{i}F_{s}^{k}\ =\ 0 ,
\end{equation}

\begin{equation}\label{Chart eq}
    \forall\,j=1\,,\,2\,,\,3\,,\ \ 
\pd{F^{j}}{s} \ =\ \sum_{k=1}^{3}F_{s}^{k}(t\,,\,s)\,A_{k}^{j}(F(t\,,\,s) ,
\end{equation}
where $A=(A_{k}^{j}) $ is the change of basis matrix:
$$u_{k}\ =\ \sum_{j=1}^{3}A_{k}^{j}\,\pdo{x^{j}} . $$

Now, depending on the complexity of the product of G, and consequently of the matrix A, finding a closed form for the solution of \eqref{Chart eq} can be unattainable. However, we do not necessarily need an explicit form of the map F in order to find the mean curvature of the surface, since, for that, it suffices to have the expressions of $F_{t}$ and $F_{s}$. With that in mind, we point to the fact that:
$$\nabla_{F_{t}}F_{s}\,-\,\nabla_{F_{s}}F_{t}\ 
=\ TF\left(\left[\pdo{t},\ \pdo{s}\right]\right) \ 
=\ TF(0)\ =\ 0 . $$

Therefore, once one obtains $F_{s}$, form \eqref{general V2 eq}, in order to find $F_{t}$, it is only necessary to solve the system:
$$\forall\,j=1\,,\,2\,,\,3\,,\ \ 
\pd{F_{s}^{j}}{t}\ -\ \pd{F_{t}^{j}}{s} \ +\ C_{ik}^{j}\,F_{t}^{i}\,F_{s}^{k} \ =\ 0 . $$

\hfill
\section{Ruled Surfaces in the Group $\Sol(3)$}
As outlined in the introduction, we first use the Koszul formula to compute the covariant derivative associated with the left-invariant metric $g_{\varepsilon,\alpha}$. Next, by applying the techniques of \cite{article 1}, Theorem \ref{Th1 Sol3 gen} characterizes all tangent frames of orthogonal ruled surfaces for $p\neq0$. This characterization leads to the explicit parameterizations given in Theorem \ref{Ruled Surf Sol(3) geral}, from which, Theorem\ref{The Min Surf Geral} identifies examples of minimal surfaces.
\subsection{Covariant Derivatives}\label{Covariant der}
From the left invariant Frame $\{e_{1}\,,\,e_{2}\,,\,e_{3}\} $, we can define the orthonormal frame
$$\begin{aligned}
    &u_{1}\ =\ \frac{1}{\varepsilon}e_{1} , 
    &u_{2}\ =\ \frac{1}{\alpha}e_{2} , 
    &&u_{3}\ =\ \frac{1}{\alpha}e_{3}
\end{aligned} . $$
\vspace{3pt}
So the commutation relations in this frame will be 
$$[u_{1}\,,\,u_{2}]\ 
=\ \frac{p}{\varepsilon}u_{2}\ +\ \frac{1}{\varepsilon}u_{3} \ ,\ \ 
[u_{1}\,,\,u_{3}]\ =\ 
-\frac{1}{\varepsilon}u_{2}\ +\ \frac{p}{\varepsilon}u_{3}\ ,\ \ 
[u_{2}\,,\,u_{3}] \ =\ 0.$$
\vspace{3pt}

With this we can calculate, via the Koszul formula, the Levi-Civita connection $\nabla$ associated to such metric is 
\vspace{1pt}

$$\begin{aligned}
	&\nabla_{u_{1}}u_{1}\ =\ 0 ,& 
	&\nabla_{u_{1}}u_{2}\ =\ \frac{1}{\varepsilon}u_{3}, & 
	&\nabla_{u_{1}}u_{3}\ =\ -\frac{1}{\varepsilon}\,u_{2}, & 
	&\nabla_{u_{2}}u_{1}\ =\ -\frac{p}{\varepsilon}\,u_{2}, & 
	&\nabla_{u_{2}}u_{2}\ =\ \frac{p}{\varepsilon}\,u_{1}, \\\\
	&\nabla_{u_{2}}u_{3}\ =\ 0,& 
	&\nabla_{u_{3}}u_{1}\ =\ -\frac{p}{\varepsilon}\,u_{3},& 
	&\nabla_{u_{3}}u_{2}\ =\ 0, & 
	&\nabla_{u_{3}}u_{3}\ =\ \frac{p}{\varepsilon}\,u_{1}.
\end{aligned} $$

\hfill
\\
\subsection{Characterization of the tangent frames of orthogonal ruled surfaces}\label{Ruled frames Sol3 gen}
Let  $F:(t_{1}\,,\,t_{2})\times (s_{1}\,,\,s_{2})\to \Sol(3) $ be a ruled surface with ruling geodesic curve $s\mapsto F(t\,,\,s) $, and let us denote:

$$dF\left(\pdo{s}\right)\, 
=\, F_{s}\, 
=\, F_{s}^{1}\,u_{1}\ +\ 
F_{s}^{2}\,u_{2}\ +\ 
F_{s}^{3}\,u_{3} 
\ \ and\ \ 
dF\left(\pdo{t}\right)\,
=\, F_{t}\, 
=\, F_{t}^{1}\,u_{1}\ +\ 
F_{t}^{2}\,u_{2}\ +\ 
F_{t}^{3}\,u_{3} .  $$

Inspired by the techniques presented in \cite{article 1}, we restrict our analysis to the orthogonal case, meaning that we impose the condition $g(F_{t}\,,\,F_{s})\,=\,0 $, in order to derive ruled surfaces. Now, an important property of the ruled orthogonal surfaces is that $\|F_{s}\|_{g} $ is a constant. Indeed, since
$$\pdo{s}\|F_{s}\|_{g}^{2} \ 
=\ \pdo{s}g(F_{s}\,,\,F_{s}) \ 
=\ 2\,g\left(\nabla_{F_{s}}F_{s}\,,\,F_{s}\right) \ 
=\ 0  \text{, and} $$
$$\pdo{t}\|F_{s}\|_{g}^{2} \ 
=\ \pdo{t}g(F_{s}\,,\,F_{s}) \ 
=\ 2\,g\left(\nabla_{F_{t}}F_{s}\,,\,F_{s}\right) \ 
=\ g\left(\nabla_{F_{s}}F_{t}\,,\,F_{s}\right) \ 
=\ \pdo{t}g(F_{t}\,,\,F_{s}) \ -\ g\left(F_{t}\,,\,\nabla_{F_{s}}F_{s}\right)  \ 
=\ 0 . $$
\\

Now, in the following proposition we characterize all possible rulling vector fields $F_{s}$ of the othogonal ruled surfaces of $\Sol(3) $.
\\
\begin{proposition}\label{Prop1 Sol3 gen}
	Let  $F:(t_{1}\,,\,t_{2})\times (s_{1}\,,\,s_{2})\to \Sol(3) $ be an orthogonal ruled surface with ruling geodesic curve $s\mapsto F(t\,,\,s) $, then $F_{s}\,=\,TF(\partial_{s}) $ can only assume two forms:
	\\\\
	(1) $$F_{s,1}\ 
	=\ u_{1} $$
	(2)
	$$F_{s,2}\ 
	=\ -\tanh \left(\omega\right) \, u_{1}\ +\ 
	\text{sech}\left(\omega\right)\,
	\cos\left(\sigma\right)\,u_{2} \ +\ 
	\text{sech}\left(\omega\right)\, 
	\sin\left(\sigma\right)\,u_{3} $$
	where $$\omega \ =\ c_{1}(t)+\frac{p s}{\varepsilon }\ , \ 
	\sigma\ =\ \frac{1}{p}\log\left(\cosh\left(\omega\right)\right) + \theta (t)\ \text{ and }\ c_{1}(t)\,,\,\theta(t) \text{ are arbitrary smooth functions. } $$
\end{proposition}
\hfil
\begin{proof}
	Since $\|F_{t}\|\,=\,1 $, let us take
	$$F_{s}^{1}\,=\,P(t\,,\,s) \ ,\ \ 
	F_{s}^{2}\,=\,Q(t\,,\,s)\,\cos(\sigma(t\,,\,s)) \ ,\ \ 
	F_{s}^{3}\,=\,Q(t\,,\,s)\,\sin(\sigma(t\,,\,s)) 
	\   \text{ , where }\ 
	P^{2}\,+\,Q^{2}\,=\,1 , $$  
	then, substituting these into the equation $\nabla_{F_{s}}F_{s}=0$ in the frame $ \{u_{1},u_{2},u_{3}\}$, we obtain the system
	$$\left\{
	\begin{aligned}
	&\pd{P}{s}\ +\  \frac{p Q^2}{\varepsilon } \ =\ 0 &(\RomanNum{1}) \\\\
	&-\frac{p}{\varepsilon } P Q\cos (\sigma ) 
	\ -\ \frac{1}{\varepsilon}P Q \sin (\sigma )\ +\ 
	\pd{Q}{s} \cos (\sigma ) \ -\ 
	Q \pd{\sigma}{s} \sin (\sigma ) \ =\ 0  &(\RomanNum{2}) \\\\
	&-\frac{ p}{\varepsilon }P Q \sin (\sigma ) 
	\ +\ \frac{1}{\varepsilon}P Q \cos (\sigma )\ +\ 
	\pd{Q}{s}\sin(\sigma)\ +\ 
	Q \pd{\sigma}{s}\cos(\sigma ) \ =\ 0  &(\RomanNum{3}) 
	\end{aligned}\right.$$
	\\\\
	(1) It is trivial that, if we choose $P\,=\,1 $ and $Q\,=\,0 $, the system above is satisfied.
	\\\\
	(2) Since $P^{2}\,+\,Q^{2}\,=\,1 $, then, from $(\RomanNum{1})$,
	$$\pd{P}{s}\ =\  \frac{p}{\varepsilon}\,\left(P^{2}\,-\,1\right) \ \ 
	\Rightarrow\ \ \frac{p}{\varepsilon}\,s \ =\ \int\frac{dP}{P^{2}-1} \ 
	=\ -\tanh^{-1}(P)\ -\ c_{1}(t)\  
	\Rightarrow\ \ P \ =\  -\,\tanh\left(\frac{p}{\varepsilon}\,s\,+\,c_{1}\right) .
	$$
	So, $$Q^{2}\,=\,1\,-\,P^{2} \ 
	=\ 1-\tanh^{2}\left(\frac{p}{\varepsilon}\,s\,+\,c_{1}\right) \ 
	=\ sech^{2}\left(\frac{p}{\varepsilon}\,s\,+\,c_{1}\right) \ \ 
	\Rightarrow\ \ Q\ =\ sech\left(\frac{p}{\varepsilon}\,s\,+\,c_{1}\right) .  $$
	
	Now, notice that
	$$P^{2}\ +\ Q^{2}\, =\, 1 \ \ 
	\Rightarrow\ \ P\,\pd{P}{s}\ +\ Q\,\pd{Q}{s}\, =\, 0 \ \ 
	\Rightarrow\ \ -\frac{p}{\varepsilon}\,P\,Q^{2}\ +\ Q\,\pd{Q}{s}\, =\, 0 \ \ 
	\Rightarrow\ \ \pd{Q}{s}\,=\,\frac{p}{\varepsilon}\,P\,Q . $$
	
	Then, plugging this value of the derivative of Q back into $(\RomanNum{2})$ and $(\RomanNum{3})$ we obtain that
	
	$$\left\{\begin{aligned}
		&-\, Q\,\sin(\sigma)\,\pd{\sigma}{s} \ -\ \frac{1}{\varepsilon}PQ\,\sin(\sigma) \ =\ 0 &(\RomanNum{4}) \\\\[-10pt]
		&Q\,\cos(\sigma)\,\pd{\sigma}{s} \ +\ \frac{1}{\varepsilon}\,P\,Q\,\cos(\sigma) \ =\ 0 &(\RomanNum{5})
	\end{aligned}\right. . $$
	\\[-10pt]
	
	Now, calculating $\ (\RomanNum{4})\,(-\sin(\sigma))\ +\ (\RomanNum{5})\,\cos(\sigma) $ we will end up with
	$$Q\pd{\sigma}{s} \ -\ \frac{q}{\varepsilon}\,P\,Q\ =\ 0 \ \ 
	\Rightarrow\ \ \pd{\sigma}{s} \ =\ -\frac{1}{\varepsilon}\,P \ 
	=\ \frac{1}{\varepsilon}\,\tanh\left(\frac{p}{\varepsilon}\,s\,+\,c_{1}\right) , $$
	finally, integrating the equation we obtain 
	$$\sigma(t\,,\,s) \ =\ \frac{1}{\varepsilon}\,\log\left(\cosh\left(\frac{p}{\varepsilon}\,s\,+\,c_{1}\right)\right) \ +\ \theta(t) $$
	
\end{proof}

For simplicity, let $M = F((t_1, t_2) \times (s_1, s_2))$ denote our parametrized surface. We now wish to define an orthonormal frame $\{F_s, Y, W\}$ along $M$ (not necessarily adapted to the surface) such that the components of $Y$ and $W$ are functions of $c_1$ and $\theta$. Since $F_t$ and $F_s$ are mutually orthogonal, $F_t$ must lie entirely within the span of $\{Y, W\}$. This conveniently reduces the number of component functions needed to express $F_t$ from three (as would be required in the standard basis $\{u_1, u_2, u_3\}$) to just two.
\\

With that objective in mind, we first notice that, if we define $Y_{1}\,,\,Y_{2}\in \Gamma_{M}(\Sol(3))$ such that $$Y_{1}=u_{2} \ \text{ and }\ Y_{2}  = \cos\left(\sigma\right)\,u_{2}\, -\, 
\sin\left(\sigma\right)\,u_{3} , \text{ with \(\sigma\) defined as in Proposition\ref{Prop1 Sol3 gen}.} $$ 

It is straightforward to verify that $Y_{1}$ and $Y_{2}$ are orthogonal to the first and second forms of $F_{s}$ from Proposition \ref{Prop1 Sol3 gen}, respectively. Consequently, to obtain the third basis element $W$, it suffices to take the cross product of $F_{s}$ and $Y$.
$$W_{1} = F_{s,1}\times Y_{1} 
= u_{3} \ \text{ and }\ 
W_{2} = F_{s,2}\times Y_{2}  
= -\text{sech}\left(\omega\right)\,u_{1}\ -\ 
\tanh \left(\omega\right)\,
\cos \left(\sigma\right)\,u_{2}\ +\ 
\tanh \left(\omega\right)\,
\sin \left(\sigma\right)\,u_{3} \\
 $$
\\[-10pt]

Since $F_{t}$ and $F_{s}$ are taken as orthogonal, we must have that $F_{t} $ lies in the space spanned by $W $ and $Y$, meaning that there exist $f_{1}\,,\,h_{1} \in C^{\infty}((t_{1}\,,\,t_{2})\times (s_{1}\,,\,s_{2})) $ such that
$$F_{t,1} \ =\ f_{1}(t\,,\,s)\,Y_{1}\ +\ h_{2}(t\,,\,s)\,W_{1} \ 
=\ f_{1}(t\,,\,s)\,u_{2}\ +\ h_{1}(t\,,\,s)\,u_{3} , $$
\\
and there exists $f_{2},h_{2}\in C^{\infty}((t_{1}\,,\,t_{2})\times (s_{1}\,,\,s_{2}))$ such that
$$ F_{t,2} \ =\ f_{2}Y_{2}\, +\, h_{2}\,W_{2}  
= -h\,\text{sech}\left(\omega\right)\,u_{1}
\, +\, 
\left(f\,\sin \left(\sigma\right) -
h\,\tanh \left(\omega\right) 
\cos \left(\sigma\right)\right) \,u_{2} 
\, -\, 
\left(h\,\tanh\left(\omega\right)\,
\sin \left(\sigma\right) +
f\,\cos \left(\sigma\right)\right)\,u_{3} $$
\\[-10pt]

Now we are in position to apply the integrability equation $\nabla_{F_{t}}F_{s} - \nabla_{F_{s}}F_{t} $, which will allow us to obtain all possible vectors $F_{t} $ and, therefore, characterize all tangent frames $\{F_{t}\,,\,F_{s}\}$ taht a parametrized orthogonal ruled surfaces in $\Sol(3)$ can possess. This is done in the following theorem.
\begin{theorem}\label{Th1 Sol3 gen}
	Let $F:(t_{1}\,,\,t_{2})\times (s_{1}\,,\,s_{2})\to \Sol(3)$ be an orthogonal ruled surface with ruling geodesics $s\mapsto F(t\,,\,s) $, then the frames $F_{t}\,=\,TF(\partial_{t}) $ and $F_{s}\,=\,TF(\partial_{s})$ must have one of the two following forms:
	\\\\
	(1) $$F_{s,1} = u_{1} 
	\ \ \ and\ \ \ 
	F_{t,1}\ =\ 
	e^{-\frac{p}{\varepsilon}s}\,\left(a_{1}(t)\,
	\cos\left(\frac{s}{\varepsilon}\right) 
	 + a_{2}(t)\,\sin\left(\frac{s}{\varepsilon}\right)\right)\,u_{2}\ +\ 
	e^{-\frac{p}{\varepsilon}s}\,
	\left(a_{2}(t)\,
	\cos\left(\frac{s}{\varepsilon}\right) 
	 - a_{1}(t)\,\sin\left(\frac{s}{\varepsilon}\right)
	\right)\,u_{3}
	 $$
	 \\
	 where $a_{1}(t)\,,\,a_{2}(t)$ are arbitrary smooth functions and the column matrix is expressed in the basis $\{u_{1}\,,\,u_{2}\,,\,u_{3}\} $.
	 \\\\
	 (2) $$F_{s,2}\ 
	 	=\ -\tanh \left(\omega\right) \, u_{1}\ +\ 
	 	\text{sech}\left(\omega\right)\,
	 	\cos\left(\sigma\right)\,u_{2} \ +\ 
	 	\text{sech}\left(\omega\right)\, 
	 	\sin\left(\sigma\right)\,u_{3} , $$ 

	 	$$F_{t,2} \ 
	 	=\ \left(\begin{gathered}
	 		-b_{2}\,-\, \frac{\varepsilon}{p}\,c_{1}'\,
	 		\tanh\left(\omega\right) 
			\\\\[-5pt]
	 		\left(-b_{2}\,\sinh\left(\omega\right) \,-\, \frac{\varepsilon}{p}\,c_{1}'\,
	 		\sinh\left(\omega\right)\,
	 		\tanh\left(\omega\right)\right)\,
	 		\cos\left(\sigma\right)\,-\, 
	 		\left(-b_{1}\,\cosh(\omega) \,+\, 
	 		\frac{\varepsilon\theta'-1}{p}\,\sinh(\omega)\right)\,\sin(\sigma) \\\\[-5pt]
	 		\left(-b_{1}\,\cosh(\omega) \,+\, 
	 		\frac{\varepsilon\theta'-1}{p}\,\sinh(\omega)\right)\,\cos\left(\sigma\right)\,-\, 
	 		\left(b_{2}\,\sinh\left(\omega\right) \,+\, \frac{\varepsilon}{p}\,c_{1}'\,
	 		\sinh\left(\omega\right)\,
	 		\tanh\left(\omega\right)\right)\,\sin(\sigma)
	 	\end{gathered}\right) $$
	 	\\
	 	where 
	 	$$\omega\ =\ \frac{p}{\varepsilon}s\,+\,c_{1} \ ,\ \ 
	 	\sigma\ =\ \frac{1}{p}\,\log\left(\cosh\left(\frac{p}{\varepsilon}s\,+\,c_{1}\right)\right) \,+\,\theta\ ,\ \ c_{1}(t)\,,\,\theta(t)\,,\,b_{1}(t)\,,\,b_{2}(t)\ \text{ are arbitrary smooth functions and } $$
	 	the column matrix is expressed in the basis $\{u_{1}\,,\,u_{2}\,,\,u_{3}\} $.
\end{theorem}

\begin{proof}
\hfil
\\\\
(1) As was proven in Proposition \ref{Prop1 Sol3 gen} and discussed above, 
$$F_{s,1}\ =\ u_{1}(t\,,\,s) \ \ \ and\ \ \ 
F_{t,1} \ 
=\ f_{1}\,u_{2}\ +\ h_{1}\,u_{3} $$

Then, expanding the integrability equation $\nabla_{F_{t}}F_{s}\,-\,\nabla_{F_{s}}F_{s}\,=\,0 $ in the basis $\{u_{1}\,,\,u_{2}\,,\,u_{3}\}$ we obtain
$$\pd{f_{1}}{s} \ +\ \frac{p}{\varepsilon}\,f_{1} \ -\ 
	\frac{1}{\varepsilon}\,h_{1} \ =\ 0 \ \text{ and }\ 
\pd{h_{1}}{s} \ +\ \frac{1}{\varepsilon}\,f_{1} \ +\ 
	\frac{p}{\varepsilon}\,h_{1} \ =\ 0 $$
\\
which is a linear equation with general solution:
$$f_{1}(t\,,\,s) \ 
	=\ e^{-\frac{p}{\varepsilon}s}\,
	\left(a_{1}(t)\,\cos\left(\frac{s}{\varepsilon}\right) 
	\ +\ a_{2}(t)\,\sin\left(\frac{s}{\varepsilon}\right)\right) \ \text{ and }\ 
h_{1}(t\,,\,s) \ 
	=\ e^{-\frac{p}{\varepsilon}s}\,
	\left(a_{2}(t)\,	\cos\left(\frac{s}{\varepsilon}\right) 
	\ -\ a_{1}(t)\,\sin\left(\frac{s}{\varepsilon}\right)\right) , $$
\\
where $a_{1}(t)\,,\,a_{2}(t)$ are arbitrary smooth functions.
\\\\
(2) As was proven in Proposition(\ref{Prop1 Sol3 gen}) and discussed above, 
$$F_{s,2}\ 
	=\ -\tanh \left(\omega\right) \, u_{1}\ +\ 
	\text{sech}\left(\omega\right)\,
	\cos\left(\sigma\right)\,u_{2} \ +\ 
	\text{sech}\left(\omega\right)\, 
	\sin\left(\sigma\right)\,u_{3} \ \text{ and }$$
\\[-10pt]
$$F_{t,2}\ =\ -h_{2}\,\text{sech}\left(\omega\right)\,u_{1}\, +\, 
f_{2}\,\sin \left(\sigma\right)\, -\, 
h_{2}\,\tanh \left(\omega\right) 
\cos \left(\sigma\right) \,u_{2} \, -\, 
h_{2}\,\tanh\left(\omega\right)\,
\sin \left(\sigma\right) -
f_{2}\,\cos \left(\sigma\right)\,u_{3}. $$
\\[-10pt]

Then, expanding the integrability equation $\nabla_{F_{t}}F_{s}\,-\,\nabla_{F_{s}}F_{s}\,=\,0 $ in the basis $\{u_{1}\,,\,u_{2}\,,\,u_{3}\}$ we obtain

$$\left\{\begin{aligned}
&\pd{h_{2}}{s}\,Q\ +\ h_{2}(t,s)\,\frac{p}{\varepsilon}\,PQ\ +\ \pd{P}{t}  &(\RomanNum{1})
\\\\
&\left( - \frac{p}{\varepsilon}h_{2} - P\pd{h}{s}  -  \pd{P}{s}h_{2}  + \frac{p}{\varepsilon}\,PQ\right)\,\cos(\sigma) + 
\left(-\frac{p}{\varepsilon}Pf_{2} + \frac{1}{\varepsilon}h_{2} - \pd{f_{2}}{s} + \frac{q}{\varepsilon}P^{2}h_{2}   + Q\frac{1}{\varepsilon}P\right)\,\sin(\sigma)  &(\RomanNum{2})
\\\\
&\left(\frac{p}{\varepsilon}Pf_{2} - \frac{1}{\varepsilon}h_{2} + \pd{f_{2}}{s} - \frac{q}{\varepsilon}P^{2}h_{2}   - \frac{1}{\varepsilon}PQ\right)\,\cos(\sigma) \ +\ 
\left(\frac{p}{\varepsilon}h_{2} + P\pd{h_{2}}{s} +  \pd{P}{s}h_{2}  - \frac{p}{\varepsilon}\,PQ\right)\,\sin(\sigma)  &(\RomanNum{3})
\end{aligned}\right. $$
\\
where 
$$P\ =\ - \tanh\left(\frac{p\,s}{\varepsilon}+ c_{1}(t)\right) \ \ ,\ \ \ 
Q\ =\ sech\left(\frac{p\,s}{\varepsilon}\ +\ c_{1}(t)\right) \ \ ,\ \ \ 
\sigma(t\,,\,s)\ 
=\ \theta(t)\ +\ \frac{1}{p}\,\log\left(\cosh\left(\frac{p\,s}{\varepsilon}\ +\ c_{1}(t)\right)\right) .  $$
\\[-7pt]

Calculating $(\RomanNum{2})\,\cos(\sigma)\,-\,(\RomanNum{3})\,\sin(\sigma) $ and $-(\RomanNum{2})\,\sin(\sigma)\,+\,(\RomanNum{3})\,\sin(\sigma)\,\cos(\sigma) $, we, respectively, get
$$
\frac{p}{\varepsilon}h_{2} + P\pd{h_{2}}{s} +  \pd{P}{s}h_{2}  - \frac{p}{\varepsilon}\,PQ\ =\ 0 \ \text{ and }\ 
\frac{p}{\varepsilon}Pf_{2} - \frac{1}{\varepsilon}h_{2} + \pd{f_{2}}{s} + \frac{1}{\varepsilon}P^{2}h_{2}   - \frac{1}{\varepsilon}PQ \ =\ 0 . $$
\\[-7pt]

Then, the general solution is
$$f_{2}(t,s)\ 
=\ \frac{1}{p}b_{2}(t)\,\sinh\left(c_{1}(t)+\frac{p s}{\varepsilon}\right) \ +\ 
b_{1}(t)\,
\cosh\left(c_{1}(t)+\frac{p s}{\varepsilon}\right) \ -\ 
\frac{\varepsilon}{p}\theta '(t)\,
\sinh\left(c_{1}(t)+\frac{p s}{\varepsilon }\right) $$
\\
$$h_{2}(t,s)\ 
=\ b_{2}(t) \cosh\left(c_{1}(t)+\frac{p s}{\varepsilon}\right)\ +\ 
\frac{\varepsilon}{p}c_{1}'(t) \sinh \left(c_{1}(t)+\frac{p s}{\varepsilon }\right) $$

\end{proof}

\subsection{The parametrized orthogonal ruled surfaces}\label{Ruled surf Sol3 gen}
As described in Section \ref{Sec 3}, we obtained the tangent frames $\{F_t, F_s\}$ using only the Lie algebra of $\Sol(3)$, without referencing its product structure. And, if one desires, this is enough to also characterize the minimal orthogonal ruled surfaces, as computing the mean curvature requires only the tangent vectors. However, our goal here is to also obtain the explicit parametrization of these surfaces in the canonical global chart. For that we  first expand our orthonormal left invariant frame in the canonical global chart of $\Sol(3) $, obtaining
$$u_{1} \ 
=\ \frac{1}{\varepsilon}\pdo{x}\,,\  u_{2}\ 
=\ \frac{1}{\alpha}e^{px}\,\cos(x)\,\pdo{y}\ +\ \frac{1}{\alpha}e^{px}\,\sin(x)\,\pdo{z} 
\, , \ 
u_{3}\ 
=\ -\frac{1}{\alpha}e^{px}\,\sin(x)\,\pdo{y}\ +\ \frac{1}{\alpha}e^{px}\,\cos(x)\,\pdo{z} . $$
\\

Then, expanding the formula for $F_{s}$ from Theorem \ref{Th1 Sol3 espec} in terms of this canonical basis yields
$$F_{s}\ 
=\ -\frac{1}{\varepsilon}\tanh\left(\omega\right)\,\pdo{x}
\ +\ 
\left( \frac{e^{p\,x(t,s)}}{\alpha}\,\text{sech}\left(\omega\right)\,
\cos\left(\sigma + x(t,s)\right) \right)\,\pdo{y}\ -\ 
\left( \frac{e^{p\,x(t,s)}}{\alpha}\,
\text{sech}\left(\omega\right)\,
\sin\left(\sigma + x(t,s) \right) 
\right)\,\pdo{z} $$
\\
and, since 
$$F_{s}\ 
=\ \pd{x}{s}\,\pdo{x}\ +\ \pd{y}{s}\,\pdo{y} \ +\ \pd{z}{s}\,\pdo{z} , $$
\\
then, by solving the system of equations, one obtains the general solution
$$F(t,s)\ 
=\ \left(
	A_{1} - 
	\frac{1}{p}\,\log(\cosh(\omega)) \ ,\ 
	A_{2} + 
	\frac{\varepsilon}{p\,\alpha}\,e^{p\,A_{1}}\,\cos\left(A_{1} + \theta\right)\,\tanh\left(\omega\right)
	\ ,\ 
	A_{3} + 
	\frac{\varepsilon}{p\,\alpha}\,e^{p\,A_{1}}\,\sin\left(A_{1} + \theta\right)\,\tanh\left(\omega\right)	
\right), $$
\\
where $A_{1}(t)\,,\,A_{2}(t)\,,\,A_{3}(t) $ are arbitrary smooth functions that come from the integration of the equations. However, as in this calculation we didn't reference the vector $F_{t}$, the family of parametrizations we have obtained is not necessarily orthogonal.  Therefore, in order to recover orthogonality of the tangent frame we have to relate the new functions, $A_{i} $'s, with the $b_{i}$'s. In order to do this, we just have to expand the derivative $\partial_{t}F$ obtained above in the basis $\{u_{1}\,,\,u_{2}\,,\,u_{3}\}$, subtract it from the expression  for $F_{t}$ found in Theorem \ref{Th1 Sol3 gen} and identify when the expression for this difference is null. Doing that for $F_{t,2} $ we get

$$A_{1}' \ =\ -\frac{1}{\varepsilon}\,b_{2}\ \text{ and }\ 
\left\{\begin{aligned}
	&A_{2}'  
	= \frac{e^{p\,A_{1}}}{\alpha}\,\left(-\,\frac{\varepsilon}{p}\,c_{1}'\,
	\cos\left(A_{1} + \theta\right) \ 
	+\ b_{1}\,\sin\left(A_{1} + \theta\right)\right) \\\\
	&A_{3}' 
	= -\frac{e^{p\,A_{1}}}{\alpha}\,\left(\frac{\varepsilon}{p}\,c_{1}'\,\sin\left(A_{1}\,+\,\theta\right)  + 
	b_{1}\,\cos\left(A_{1} + \theta\right)\right) .
\end{aligned}\right. $$
\\

Now, using the first formula for $F_{s}$ from Proposition \ref{Prop1 Sol3 gen} we will have
$$F_{s}\ 
=\ \frac{1}{\varepsilon} \pdo{x}\ \text{, meaning that }\ 
F(t\,,\,s)\ =\ \left(A_{1}(t)\ +\ \frac{1}{\varepsilon}\,s\ ,\ 
A_{2}(t)\ ,\ A_{3}(t)\right) . $$
\\[-7pt]

Repeating the previous process, the relations between the $A_{i}$'s and $a_{i}$'s will be
$$A_{1}'\ =\ 0 \ \text{ and }\ 
\left\{\begin{aligned}
	&A_{2}'\ =\ \frac{1}{\alpha}e^{p\,A_{1}}
	\left(a_{1}\,\cos\left(\frac{2}{\varepsilon}s\,-\,A_{1}\right) \ +\ 
	a_{2}\,\sin\left(\frac{2}{\varepsilon}s\,-\,A_{1}\right)\right) \\\\
	&A_{3}' \ =\ \frac{1}{\alpha}e^{p\,A_{1}}
	\left(a_{2}\,\cos\left(\frac{2}{\varepsilon}s\,-\,A_{1}\right) \ -\ 
	a_{1}\,\sin\left(\frac{2}{\varepsilon}s\,-\,A_{1}\right)\right) .
\end{aligned}\right.\, $$

From this system of equations and the expression of $F(t\,,\,s)$ it is clear that the only requirement to get the orthogonal surfaces is imposing $A_{1}'$ to be identically zero, meaning $A_{1}=R\in \mathbb{R} $.
\\

We summarize the results obtained in the following theorem:
\begin{theorem}\label{Ruled Surf Sol(3) geral}
    Let $F:(t_{1}\,,\,t_{2})\times (s_{1}\,,\,s_{2})\to \Sol(3)$ be an orthogonal ruled surface, then it must have one of the following two forms:
    \\
    (1) $$F(t,s)\ 
	=\ \scalebox{0.95}{\(\left(\begin{gathered}
			A_{1}(t)\ -\ 
			\frac{1}{p}\,\log\left(\cosh\left(\frac{p}{\varepsilon}s + c_{1}\right)\right) \\\\
			\int_{t_{0}}^{t}e^{p A_{1}}\,
			\left(-\frac{\varepsilon}{p \alpha}  c_{1}' 
				\cos\left(A_{1} + \theta\right) 
				+  \frac{b_{1}}{\alpha}   \sin\left(A_{1} + \theta\right)\right)\,dt
			 + 
			\frac{\varepsilon}{p \alpha} e^{p\,A_{1}} \cos\left(A_{1} + \theta\right) \tanh\left(\frac{p}{\varepsilon}s + c_{1}\right)
			\\\\
			-\int_{t_{0}}^{t}e^{p A_{1}}\,\left(\frac{\varepsilon}{p\alpha}  c_{1}' 
				\sin\left(A_{1} + \theta\right)   +
				\frac{b_{1}}{\alpha} \cos\left(A_{1}\,+ \theta\right)\right)\,dt
			 + 
			\frac{\varepsilon}{p \alpha} e^{p A_{1}} \sin\left(A_{1} + \theta\right) \tanh\left(\frac{p}{\varepsilon}s + c_{1}\right)
		\end{gathered}\right)\)}  , $$
\\
where $c_{1}(t)\,.\,\theta(t)\,,\,A_{1}(t)\,,\,b_{1}(t) $ are arbitrary smooth functions.
\\\\
(2)$$F(t\,,\,s)\ =\ \left(R\ +\ \frac{s}{\varepsilon}\ ,\ 
A_{2}(t)\ ,\ A_{3}(t)\right) ,$$
\\
where $R\in \mathbb{R}$ and $A_{2}(t)\,,\,A_{3}(t) $ are arbitrary smooth functions.

\end{theorem}

\subsection{The Minimal Orthogonal Ruled Surfaces}\label{Min Sol3 gen}
All we have done so far is to find the ruled surface, not necessarily minimal, so now we have to find the relations that such surfaces must satisfy in order to achieve minimality. For that we use the values of $F_{t} $ and $F_{s}$ obtained in Theorem(\ref{Th1 Sol3 gen}) and the following expression for the mean curvature of orthogonal ruled surfaces:
$$	H \ 
=\  \frac{1}{\|F_{t}\wedge F_{s}\|_{g}^{3}}\,
g\left(\nabla_{F_{t}}F_{t} \ ,\  F_{t}\times F_{s}\right) . $$

In the first case of Theorem(\ref{Th1 Sol3 gen}) we have
$$\|F_{t}\times F_{s}\|_{g}^{3}H \ 
=\ -\,e^{-\frac{2p}{\varepsilon}s}\,
\left|\begin{matrix}
    a_{1} & a_{2} \\
    a_{1}' & a_{2}'
\end{matrix}\right| \ \ \ and\ \ \ 
\|F_{t}\times F_{s}\|_{g}^{2}\ 
=\ e^{-\frac{2p}{\varepsilon}s}\,\left(a_{1}^{2}\,+\, a_{2}^{2}\right) .$$
\\[-10pt]

For the right-hand side of the first equation to vanish while maintaining $\|F_{t}\times F_{s}\| \neq 0$, we require either that $a_2$ vanishes while $a_1$ does not, or that $a_1 = k a_2$ for some constant $k \in \mathbb{R}$, where $a_2$ is a nowhere-vanishing function. Substituting these restrictions for $a_1$ and $a_2$ into the parameterizations given in Theorem \ref{Ruled Surf Sol(3) geral}, we, respectively, obtain the following surfaces:

$$F(t,s)\ 
=\ \left(\frac{1}{\varepsilon}\,s\ +\ R\ ,\ 
\frac{1}{\alpha}e^{p\,R}\,
		\cos\left(\frac{2}{\varepsilon}\,s\,-\,R\right)\,f(t)\ ,\ 
-\frac{1}{\alpha}e^{p\,R}\,
		\sin\left(\frac{2}{\varepsilon}\,s\,-\,R\right)\,f(t)
\right)\ \text{ and } $$

$$F(t,s)\ 
=\ \left(
\frac{s}{\varepsilon} + R\ ,\ 
\frac{e^{p\,R}\,f(t)}{\alpha}
	\left(k\,\cos\left(\frac{2s}{\varepsilon} - R\right)  + 
	\sin\left(\frac{2s}{\varepsilon} - R\right)\right) \ ,\ 
\frac{e^{p\,R}\,f(t)}{\alpha} 
	\left(\cos\left(\frac{2s}{\varepsilon} - R\right)  -  k\,\sin\left(\frac{2s}{\varepsilon} - R\right)\right)
\right). $$
where $f(t)=\int a_{2}(r)\,dr $.
\\

Now we turn our attention to the second case of Theorem \ref{Th1 Sol3 gen}. In this situation we will have that
\begin{equation}\label{Min Cond Sol3 ort gen}
	\begin{aligned}
		&\|F_{t}\times F_{s}\|_{g}^{3}\,H\ = \\\\[-7pt]
		&\left(\frac{b_{2}b_{1}' - b_{2}'b_{1}}{2}  - 
		\varepsilon \frac{b_{2}c_{1}'' - b_{2}'c_{1}'}{2p^{2}}  + 
		\frac{\varepsilon b_{1}}{2p}\left((c_{1}')^{2} + (\theta')^{2}\right)  + 
		\varepsilon^{2}\frac{c_{1}''\theta' -c_{1}'\theta''}{2p^{2}}
		+ \frac{(p^{2}+1)b_{1}b_{2}^{2}}{2\varepsilon p}  + 
			    \frac{p b_{1}^{3}}{2\varepsilon}  -  
			    \frac{b_{2}b_{1}\theta'}{p}
			    \right)\,\cos(2\omega)  + \\\\[-5pt]
		&\left(\frac{\varepsilon}{2p}\left(\left(b_{2}'\theta' - b_{2}\theta''\right) \ +\  
			    \left(b_{1}'c_{1}' - b_{1}c_{1}''\right)\right)  + 
			    \frac{1}{\varepsilon}\left(b_{2}b_{1}^{2}\right)  + 
			    b_{1}b_{2}c_{1}' \ -\ b_{1}^{2}\theta'
			    \right)\,\sin(2\omega) +
		\left(\frac{b_{2}b_{1}' - b_{1}b_{2}'}{2}  - 
					    \frac{\varepsilon }{2p^{2}}\left(b_{2}'c_{1}' - b_{2}c_{1}''\right)  + 
					     \right. \\\\[-5pt]
		&\left. \frac{b_{2}b_{1}\theta' + b_{2}^{2}c_{1}'}{p}  +
		\frac{p^{2}-1}{2\varepsilon p}\left(b_{2}^{2}b_{1}\right)  - 
			    \frac{\varepsilon}{p}b_{2}c_{1}'\theta' -\ 
			    	    \frac{\varepsilon}{2p}\,b_{1}\,\left((\theta')^{2} + 3(c_{1}')^{2}\right) \ +\ 
			    	    \frac{p}{2\varepsilon}b_{1}^{3} \ +\ 
			    	    \frac{\varepsilon^{2}}{2p^{2}}\left(c_{1}'\theta'' - c_{1}''\theta'\right)\right) .
	\end{aligned}
\end{equation}
\\[-10pt]
Given the difficulty of solving such an equation, we shall simplify it by taking $ b_{1}=0 $. Doing that the expression above becomes
\\[-10pt]
$$\begin{aligned}
    &&\|F_{t}\times F_{s}\|^{3}_{g}\,H\ = 
    &-\left(\frac{\varepsilon}{2p^{2}}\,\left(b_{2}\,c_{1}''\,-\,b_{2}'\,c_{1}'\right) \ +\ 
    \frac{\varepsilon^{2}}{2p^{2}}\,\left(c_{1}''\,\theta'\,-\,c_{1}'\,\theta''\right)
    \right)\,\cos\left(2\,\omega\right) \ +\ 
    \frac{\varepsilon}{2p}\,\left(b_{2}'\,\theta'\,-\,b_{2}\,\theta''\right)\,
    \sin\left(2\,\omega\right) \ +\  \\\\
    &&&-\ \frac{b_{2}\,c_{1}'}{p}\left(-b_{2} \, +\, \varepsilon\,\theta'\right) \ -\ 
    \frac{\varepsilon}{2p^{2}}\,\left(b_{2}'\,c_{1}'\ -\ b_{2}\,c_{1}''\right) \ +\ 
    \frac{\varepsilon^{2}}{2p^{2}}\,\left(c_{1}'\,\theta''\ -\ \theta'\,c_{1}''\right) .
\end{aligned} $$
\\[-10pt]

Since $\cos(2\omega) $ and $\sin(2\omega) $ are linearly independent functions, in order for the expression for $\|F_{t}\times F_{s}\|^{3}_{g}\,H $ to vanish, we must require that the coefficients of both functions, together with the independent term to be identically zero. So, subtracting the term in front of $\cos(2\omega)$ from the independent one we get the following system of differential equations
$$\left\{\begin{aligned}
	&-\frac{\varepsilon}{2p^{2}}\,\left(b_{2}\,c_{1}''\,-\,b_{2}'\,c_{1}'\right) \ +\ 
	\frac{\varepsilon^{2}}{2p^{2}}\,\left(c_{1}''\,\theta'\,-\,c_{1}'\,\theta''\right)\ =\ 0 \ &(1) \\\\
	&b_{2}'\,\theta'\,-\,b_{2}\,\theta'' \ =\ 0 \ &(2)
	\\\\
	&\frac{b_{2}\,c_{1}'}{p}\left(-b_{2} \, +\, \varepsilon\,\theta'\right)  \ =\ 0 \ &(3) .
\end{aligned}\right.  $$
\\[-10pt]

Observing equation (3), we can see that there are three situations in which it vanishes: 
$$\theta' = \frac{1}{\varepsilon}\,b_{2}\ ,\ c_{1}\text{ is a constant function or }b_{2}\text{ vanishes everywhere.} $$

In what follows we study these three cases separately.
\\\\
\textbf{(Case \RomanNum{1})}\label{Case 1 geral} 
If $\ \theta'= (1/\varepsilon)\,b_{2} $, then not only we satisfy equation (3), but we also satisfy equations (1) and (2). Indeed, substituting our hypothesis into these equations yields
$$-\frac{\varepsilon}{2p^{2}}\,\left(b_{2}\,c_{1}''\,-\,b_{2}'\,c_{1}'\right) \ +\ 
\frac{\varepsilon^{2}}{2p^{2}}\,\left(c_{1}''\,\theta'\,-\,c_{1}'\,\theta''\right) \ =\ -\frac{\varepsilon}{2p^{2}}\,\left(b_{2}\,c_{1}''\,-\,b_{2}'\,c_{1}'\right) \ +\ 
\frac{\varepsilon }{2p^{2}}\,\left(c_{1}''\,b_{2}\,-\,c_{1}'\,b_{2}'\right) \ 
=\ 0\  \text{ and } $$

$$b_{2}'\,\theta'\,-\,b_{2}\,\theta'' \ 
=\ \frac{1}{\varepsilon}\left( b_{2}b_{2}'\,-\, b_{2}'b_{2}\right) \ =\ 0 . $$
\\[-10pt]

Plugging such restriction in the parametrization found in Theorem\ref{Ruled Surf Sol(3) geral}, we obtain the surface
$$F(t\,,\,s)\ =\ 
\left(
	A_{1}(t) \ -\ \frac{1}{p}\,\log\left(\cosh\left(\frac{p}{\varepsilon}s\,+\,c_{1}(t)\right)\right)  \ ,\ 
	\frac{\varepsilon}{p\alpha}\,e^{p\,A_{1}(t)}\,
	\tanh\left(\frac{p}{\varepsilon}s\,+\,c_{1}(t)\right) \ ,\ 
	-\,\frac{\varepsilon}{p\alpha}\int_{t_{0}}^{t}e^{p\,A_{1}(t)}\,c_{1}'(t)\,dt\right). $$

Also, since $$\|F_{t}\times F_{s}\|^{2}_{g}\ 
=\ \left(b_{2}\,\cosh\left(\frac{p}{\varepsilon}s\,+\,c_{1}(t)\right)\ +\ 
\frac{\varepsilon}{p}\,\sinh\left(\frac{p}{\varepsilon}s\,+\,c_{1}(t)\right) \right)^{2} , $$
in order for our surface to not degenerate into a curve, we will also require $b_{2}$ to be a non vanishing function. 
\\\\
\textbf{(Case \RomanNum{2})}\label{Case 2 geral} Suppose $c_{1}(t) = k \in \mathbb{R}$ for all $t \in (t_{1}, t_{2})$. While Equation (1) is trivially satisfied, Equation (2) requires either $\theta' = a\,b_{2}$ for some constant $a \in \mathbb{R}$, or that $b_{2}$ vanishes identically. This demonstrates that assuming $c_1$ to be constant relaxes the constraints on the system, indeed, rather than being restricted to $a = 1/\varepsilon$ as in Case I, the parameter $a$ can be an arbitrary real number. Consequently, we divide our analysis into two sub-cases:
\\\\
\textbf{(\RomanNum{2}.1)}\label{Case 2-1 geral} If $\theta'\,=\,a\,b_{2}\ $, where $a\in \mathbb{R} $, then the parametrization from Theorem \ref{Ruled Surf Sol(3) geral} gives the surface
$$F(t\,,\,s)\ =\ 
\left(
	A_{1} \ -\  \frac{1}{p}\,\log\left(\cosh\left(\omega_{2}\right)\right) \ ,\ 
	\frac{\varepsilon}{p\alpha}\,e^{p\,A_{1}}\,
	\cos\left((a\,\varepsilon-1)\,A_{1}\right)\,
	\tanh\left(\omega_{2}\right)\ ,\ 
	-\frac{\varepsilon}{p\alpha}\,e^{p\,A_{1}}\,
	\sin\left((a\,\varepsilon-1)\,A_{1}\right)\,
	\tanh\left(\omega_{2}\right)
	\right),
 $$
 \\
 where $\omega_{2}= (p/\varepsilon)s + k $.

Also, since
$$\|F_{t}\times F_{s}\|_{g}^{2} \ 
=\ \frac{b_{2}^{2}}{2p^{2}}\,\left(p^{2}\ -\ (a\varepsilon-1)^{2}\ +\ 
\left(p^{2}+ (a\varepsilon-1)^{2}\right)\,\cosh(2\,\omega_{2})\right) , $$
\\
we must require $b_{2}$ to be a non vanishing function. 
\\\\
\textbf{(\RomanNum{2}.2)}\label{Case 2-2 geral} If $b_{2}$ is a null function, then the surface will be 
$$F(t\,,\,s)\ =\ 
\left(
	-\,\frac{1}{p}\,
			\log\left(\cosh\left(\omega_{3}\right)\right) \ ,\ 
	\frac{\varepsilon}{p\alpha}\,\cos(\theta(t))\,
			\tanh\left(\omega_{2}\right) \ ,\ 
	\frac{\varepsilon}{p\alpha}\,\sin(\theta(t))\,
			\tanh\left(\omega_{2}\right) 
\right) . $$
\\[-10pt]

Also, since 
$$\|F_{t}\times F_{s}\|_{g}^{2} \ 
=\ \frac{\varepsilon^{2}}{p^{2}}\,(\theta')^{2}\,\sinh(\omega_{2}) , $$
\\
we have to require that $\theta' $ is a nowhere vanishing function. 
\\\\
\textbf{(Case \RomanNum{3})}\label{Case 3 geral} If $b_{2}$  vanishes for all $t\in (t_{1}\,,\,t_{2}) $ , then equation (1) becomes $c_{1}''\theta'\,-\,c_{1}'\theta''\,=\,0 , $ which is satisfied if $\theta' = l\,c_{1}'$ for some $l\in \mathbb{R} $ or if $c_{1}$ is a constant function. However if both $c_{1}'$ and $b_{2}$ are identically null, we fall back in the second sub-case of  \textbf{(Case \RomanNum{2})}, therefore, we only need to analyze the situation where $\theta'= l\,c_{1}' $. Under this assumption, the parametrization found in Theorem \ref{Ruled Surf Sol(3) geral} gives the surface

$$F(t\,,\,s)\ =\ 
\left(
	-\,\frac{1}{p}\,
	\log\left(\cosh\left(\omega\right)\right)  \ ,\ 
	\frac{\varepsilon}{p\alpha}\,\cos(l\,c_{1}(t))\,
		\left(\frac{1}{l}\ +\ \tanh\left(\omega\right)\right) \ ,\ 
	\frac{\varepsilon}{p\alpha}\,\sin(l\,c_{1}(t))\,
		\left(-\frac{1}{l}\ +\ \tanh\left(\omega\right)\right)
\right) . $$

Also, since $$\|F_{t}\times F_{s}\|_{g}^{2}\ 
=\ \frac{\varepsilon^{2}}{p^{2}}\,(1+l^{2})\,(c_{1}')^{2}\,\sinh^{2}\left(\omega\right) , $$
\\
we need to require $c_{1}'$ to be a nowhere vanishing function.
\\\\[-10pt]

We now go back to the situation where $b_{1}$ is not identically zero and apply the conditions found in cases (\RomanNum{1}), (\RomanNum{2}) and (\RomanNum{3}) in order to simplify the equation \eqref{Min Cond Sol3 ort gen}. Assuming the hypothesis of \textbf{(Case \RomanNum{1})}, in other words that $\theta'= b_{2}/\varepsilon  $, we will have
$$\begin{aligned}
    &&\|F_{t}\times F_{s}\|^{3}H \ = &\left(\frac{p}{2\varepsilon}\,b_{1}^{3}\ +\ 
	\frac{p}{2\varepsilon}b_{1}b_{2}^{2}\ +\ 
	\frac{1}{2}\left(b_{2}b_{1}'\,-\,b_{1}b_{2}'\right) \ +\ 
	\frac{\varepsilon}{2p}b_{1}(c_{1}')^{2}\right)\,\cos(2\omega) \ +\ \\\\
    &&&+\ \left(b_{1}b_{2}c_{1}'\ +\ \frac{\varepsilon}{2p}\,\left(b_{1}'c_{1}'\,-\,b_{1}c_{1}''\right)\,
	\right)\,\sin(2\omega) \ +\ \frac{p}{2\varepsilon}\,b_{1}^{3}\ +\ 
	\frac{p}{2\varepsilon}b_{1}b_{2}^{2}\ +\ 
	\frac{1}{2}\left(b_{2}b_{1}'\,-\,b_{1}b_{2}'\right) \ -\ 
	\frac{3\varepsilon}{2p}b_{1}(c_{1}')^{2} .
\end{aligned}  $$
\\[-10pt]

Subtracting the coefficient of $\cos(2\omega)$ from the independent term we obtain $b_{1}\,(c_{1}')^{2} = 0 $, which implies that $b_{1}$ or $c_{1}$ is identically null.
Since we are analyzing the cases where $b_{1}$ doesn't vanishes everywhere, then we take $c_{1}'$ identically null, meaning that $c_{1}=k $ for some $k\in \mathbb{R}$. So, we just have to solve the equation
$$\frac{p}{2\varepsilon}\,b_{1}^{3}\ +\ 
\frac{p}{2\varepsilon}b_{1}b_{2}^{2}\ +\ 
\frac{1}{2}\left(b_{2}b_{1}'\,-\,b_{1}b_{2}'\right) \ =\ 0  $$
\\[-10pt]

Seeing this as an autonomous equation for $b_{2}$, the general solution is
$$b_{2}(t)\ 
=\ b_{1}\,\tan\left(a\ +\ \frac{p}{\varepsilon}\,
B_{1}(t)\right) \ \text{ , where } a\in \mathbb{R} \text{ and } 
B_{1}(t)\ =\ \int_{t_{0}}^{t}b_{1}(u)\,du . $$
\\[-10pt]

 Substituting the restrictions above in the parametrization found in Theorem \ref{Ruled Surf Sol(3) geral}, we obtain the surface

$$F(t\,,\,s)\ =\ 
\left(-\frac{G(t)}{\varepsilon}\ -\ 
\frac{1}{p}\,\log\left(\cosh\left(\frac{p}{\varepsilon}\,s + k \right)\right) \ ,\ 
-\frac{1}{\alpha}\,\int_{t_{0}}^{t}b_{1}(t')\,
e^{p\,G(t')}\,dt' \ +\ 
\frac{\varepsilon}{p\alpha}\,e^{p\,G(t)}\,
\tanh\left(\frac{p}{\varepsilon}\,s + k  \right) \ ,\ 0
\right) , $$
where 
$$G(t)\ 
=\ -\frac{1}{\varepsilon}\int_{t_{0}}^{t}b_{1}(r)\,\tan\left(a\ +\ \frac{p}{\varepsilon}\,
B_{1}(t)\right)\,dr . $$
\\[-10pt]

Also, since $\|F_{t}\times F_{s}\|^{2}_{g} \ 
=\ b_{1}^{2}\,\cosh^{2}(\omega_{2})\,\sec^{2}\left(a\ +\ \frac{p}{\varepsilon}\,
B_{1}(t)\right) , $ we have to require $b_{1}$ to be nowhere vanishing.
\\

Assuming the conditions of the first sub-case of \textbf{(Case \RomanNum{2})}, that is $c_{1} = k\in \mathbb{R}$ and $ \theta' = a\,b_{2} $, then we need to solve the system
$$\left\{\begin{aligned}
	&\frac{p}{2\varepsilon}b_{1}^{3}\ +\ 
	\frac{1}{2}b_{2}b_{1}'\ +\ \frac{1}{2p\varepsilon}b_{1}\left(\,\left(p^{2}+(a\varepsilon-1)^{2}\right)\,b_{2}^{2}\,-\,p\varepsilon b_{2}'\right)\ =\ 0 \\\\
	&\left(a-\frac{1}{\varepsilon}\right)\,b_{1}^{2}\,b_{2}\ =\ 0 \\\\
	&\frac{p}{\varepsilon}b_{1}^{3}\,+\, \frac{p}{\varepsilon}b_{1}b_{2}^{2}\,+\,b_{2}b_{1}'\,-\,b_{1}b_{2}'\ =\ 0
\end{aligned}\right. $$
\\[-10pt]

In order to annihilate the second equation of the system we need that $a\,=\,1/\varepsilon$ or $b_{2}\,=\,0$. However, in order to not repeat \textbf{(Case \RomanNum{1})}, we must take $b_{2}$ identically zero. But, even by imposing that we end up with $b_{1}=0 $. Therefore, this case doesn't yield a minimal surface.
\\[-10pt]

Assuming the conditions of the second sub-case of \textbf{(Case \RomanNum{2})}, being $c_{1}\,=\,k\in \mathbb{R}$ and $b_{2}$ identically zero, then we need to solve
$$\frac{p}{2\varepsilon}b_{1}^{3}\,+\,\frac{p}{2\varepsilon}b_{1}(\theta')^{2}\ =\ 0 \ ,\ b_{1}^{2}\theta'\ =\ 0 \ ,\ 
\frac{p}{\varepsilon}b_{1}^{3}\ =\ 0 $$
\\
which is only satisfied when $b_{1}$ is identically null. Therefore, this case also doesn't yield a minimal surface.
\\

Assuming the conditions of \textbf{(Case \RomanNum{3})}, being $b_{2}$ identically null and $\theta'= l\,c_{1}' $ for some $l\in \mathbb{R} $, we must solve the system
\\
$$\left\{\begin{aligned}
    &\frac{p}{2\varepsilon}b_{1}^{3} \ +\ 
    \frac{\varepsilon}{2p}(1+l^{2})b_{1}(c_{1}')^{2}\ =\ 0 \\\\
    &-l\,b_{1}^{2}c_{1}'\ +\ \frac{\varepsilon}{2p}\,b_{1}'c_{1}' \ -\ 
    \frac{\varepsilon}{2p}\,b_{1}c_{1}'' \ =\ 0 \\\\
    &\frac{p}{2\varepsilon}b_{1}^{3} \ +\ 
    \frac{\varepsilon}{2p}(1+l^{2})b_{1}(c_{1}')^{2}\ =\ 0
\end{aligned}
\right. $$
\\[-10pt]

From the first equation above, 
$$b_{1}\,\left(\frac{p}{2\varepsilon}b_{1}^{2} \ +\ 
\frac{\varepsilon}{2p}(1+l^{2})(c_{1}')^{2}\right)\ =\ 0 \ \ 
\Rightarrow\ \ \frac{p}{2\varepsilon}b_{1}^{2} \ +\ 
\frac{\varepsilon}{2p}(1+l^{2})(c_{1}')^{2} \ =\ 0 \ \ 
\Leftrightarrow\ \ b_{1}\,=\,0 \ \text{ and }\ c_{1}'\,=\,0 .$$
\\[-10pt]

Therefore, this case does not yield a minimal surface, since, here, we are assuming that $b_{1}$ is not identically null. Finally, we summarized the results obtained in the following theorem:
\\[-10pt]

\begin{theorem}\label{The Min Surf Geral}
	The following are orthogonal ruled minimal surfaces in $\Sol(3)$:
	\\\\
	(1.1)
	$$F(t\,,\,s)\ =\ 
	\left(\frac{s}{\varepsilon}\ +\ R \ ,\ 
	\frac{1}{\alpha}e^{p\,R}\,
			\cos\left(\frac{2}{\varepsilon}\,s\,-\,R\right)\,f(t)\ ,\ 
	-\frac{1}{\alpha}e^{p\,R}\,
			\sin\left(\frac{2}{\varepsilon}\,s\,-\,R\right)\,f(t)
	\right) , $$
	where $ R\in \mathbb{R}\ \text{ and } f(t)\ \text{ is an arbitrary smooth function }$ whose derivative never vanishes.
	\\\\
	(1.2) $$F(t,s)\ 
	=\ \left(
	\frac{s}{\varepsilon} + R\ ,\ 
	\frac{e^{p\,R}}{\alpha}\,f(t)
		\left(k\,\cos\left(\frac{2s}{\varepsilon} - R\right)  + 
		\sin\left(\frac{2s}{\varepsilon} - R\right)\right) \ ,\ 
	\frac{e^{p\,R}}{\alpha}\,f(t) 
		\left(\cos\left(\frac{2s}{\varepsilon} - R\right)  -  k\,\sin\left(\frac{2s}{\varepsilon} - R\right)\right)
	\right),  $$
	\\
	$\text{ where } R\in \mathbb{R}\ \text{ and } f(t)\ \text{ is an arbitrary } $ smooth function whose derivative never vanishes.
	\\\\
	(2.1) $$F(t\,,\,s)\ =\ 
	\left(
		A_{1}(t) \ -\ \frac{1}{p}\,\log\left(\cosh\left(\frac{p}{\varepsilon}s\,+\,c_{1}(t)\right)\right)  \ ,\ 
		\frac{\varepsilon}{p\alpha}\,e^{p\,A_{1}(t)}\,
		\tanh\left(\frac{p}{\varepsilon}s\,+\,c_{1}(t)\right) \ ,\ 
		-\,\frac{\varepsilon}{p\alpha}\int_{t_{0}}^{t}e^{p\,A_{1}(t)}\,c_{1}'(t)\,dt\right),  $$
	\\
	$\text{ where }c_{1} \text{ is an arbitrary smooth function and} $ $A_{1}(t)$ is a smooth function whose derivative never vanishes.
	\\\\
	(2.2) $$F(t\,,\,s)\ =\ 
	\left(
		A_{1} \ -\  \frac{1}{p}\,\log\left(\cosh\left(\omega_{2}\right)\right) \ ,\ 
		\frac{\varepsilon}{p\alpha}\,e^{p\,A_{1}}\,
		\cos\left((a\,\varepsilon-1)\,A_{1}\right)\,
		\tanh\left(\omega_{2}\right)\ ,\ 
		-\frac{\varepsilon}{p\alpha}\,e^{p\,A_{1}}\,
		\sin\left((a\,\varepsilon-1)\,A_{1}\right)\,
		\tanh\left(\omega_{2}\right)
		\right), $$
	\\
	$\text{ where } \omega_{2}=(p/\varepsilon)s+k\ ,\ a\,,\,k\in \mathbb{R} \text{ and } A_{1}(t)  $ is an arbitrary smooth function whose derivative never vanishes.
	\\\\
	(2.3) $$F(t\,,\,s)\ =\ 
	\left(
		-\,\frac{1}{p}\,
				\log\left(\cosh\left(\omega_{2}\right)\right) \ ,\ 
		\frac{\varepsilon}{p\alpha}\,\cos(\theta(t))\,
				\tanh\left(\omega_{2}\right) \ ,\ 
		\frac{\varepsilon}{p\alpha}\,\sin(\theta(t))\,
				\tanh\left(\omega_{2}\right)
	\right),  $$
	\\
	$\text{ where } \omega_{2}=(p/\varepsilon)s+k\ ,\  k\in \mathbb{R} \text{ and } \theta(t) \text{ is an arbitrary smooth function } $whose derivative never vanishes.
	\\\\
	(2.4) $$F(t\,,\,s)\ =\ 
	\left(
		-\,\frac{1}{p}\,
		\log\left(\cosh\left(\omega\right)\right)  \ ,\ 
		\frac{\varepsilon}{p\alpha}\,\cos(l\,c_{1}(t))\,
			\left(\frac{1}{l}\ +\ \tanh\left(\omega\right)\right) \ ,\ 
		\frac{\varepsilon}{p\alpha}\,\sin(l\,c_{1}(t))\,
			\left(-\frac{1}{l}\ +\ \tanh\left(\omega\right)\right)
	\right) , $$
	\\
	$\text{ where } \omega = (p/\varepsilon)s+c_{1}(t)\ ,\  l\in \mathbb{R}\setminus\{0\} \text{ and } c_{1}(t) \text{ is an arbitrary smooth function } $ whose derivative never vanishes.
	\\\\
	(2.5)$$F(t\,,\,s)\ =\ 
	\left(-\frac{G(t)}{\varepsilon}\ -\ 
	\frac{1}{p}\,\log\left(\cosh\left(\frac{p}{\varepsilon}\,s + k \right)\right) \ ,\ 
	-\frac{1}{\alpha}\,\int_{t_{0}}^{t}b_{1}(t')\,
	e^{p\,G(t')}\,dt' \ +\ 
	\frac{\varepsilon}{p\alpha}\,e^{p\,G(t)}\,
	\tanh\left(\frac{p}{\varepsilon}\,s + k  \right) \ ,\ 0
	\right) , $$
	where 
	$$G(t)\ 
	=\ -\frac{1}{\varepsilon}\int_{t_{0}}^{t}b_{1}(r)\,\tan\left(a\ +\ \frac{p}{\varepsilon}\,
	\int_{t_{0}}^{r}b_{1}(u)\,du\right)\,dr $$
	\\
	where $a\,,\,k\,\in \mathbb{R} $, $b_{1}(t) $ is a non vanishing smooth function.
	\\\\
	Here, labels (1.1) and (1.2) correspond to the first solution for $F_{s}$ in Proposition~\ref{Prop1 Sol3 gen}, while (2.1) to (2.5) refer to the second solution found in the same proposition
\end{theorem}

\hfill
\\
\section{Minimal Ruled Surfaces in the Group $\Sol(3)$ with $p=0$}
In our discussion of ruled surfaces in $\Sol(3)$ the parameter $p$ appears in the denominator of certain terms in both $F_{t}$ and $F_{s}$, so it cannot be taken to be zero. Therefore, in what follows, we examine the ruled surfaces corresponding to the case where $p=0 $.
\\
\subsection{Covariant derivatives for $p=0$}
Using the formulas from Section \ref{Covariant der} for $p=0$ the commutation relations and connection coefficients become
$$[u_{1}\,,\,u_{2}]\ 
=\  \,\frac{1}{\varepsilon}u_{3} 
\ \ ,\ \ \ 
[u_{1}\,,\,u_{3}]\ =\ 
-\frac{1}{\varepsilon}u_{2}
\ \ ,\ \ \ 	
[u_{2}\,,\,u_{3}] \ =\ 0 \ ,\, and\ \ $$

$$\begin{aligned}
	&\nabla_{u_{1}}u_{1} = 0, & 
	&\nabla_{u_{1}}u_{2} = \frac{1}{\varepsilon}u_{3}, & 
	&\nabla_{u_{1}}u_{3} = -\frac{1}{\varepsilon}\,u_{2}, & 
	&\nabla_{u_{2}}u_{1} = 0, & 
	&\nabla_{u_{2}}u_{2} = 0, &
	&\nabla_{u_{2}}u_{3} = 0, &
	&\nabla_{u_{3}}u_{1} = 0, \\\\
	&\nabla_{u_{3}}u_{2} = 0, & 
	&\nabla_{u_{3}}u_{3} = 0.
\end{aligned}  $$
\\

\subsection{The Ruled Surfaces for $p=0$}
Using the procedure detailed in Section \ref{Sec 3}, the following theorem characterizes the possible tangent frames, $F_s$ and $F_{t}$, of the ruled surfaces in $\Sol(3)$.
\\

\begin{theorem}\label{Th1 Sol3 espec} Let  $F:(t_{1}\,,\,t_{2})\times (s_{1}\,,\,s_{2})\to Sol(3) $ be a ruled surface with ruling geodesic curve $s\mapsto F(t\,,\,s) $, then 
\\
$$F_{s}\ 
=\ c_{1}(t)\,u_{1} \ +\ 
\left(c_{2}(t) \cos \left(\phi\right) + 
c_{3}(t) \sin \left(\phi\right)\right)\,u_{2} \ +\ 
\left(-c_{2}(t) \sin \left(\phi\right) + 
c_{3}(t) \cos \left(\phi\right)\right)\,u_{3}
  \text{ and } $$
\\
$$F_{t}\ 
=\ \left(
\begin{gathered}
	b_{1}(t)+s c_{1}'(t) 
	\\\\
	\cos\left(\phi\right)
	\left(\left( c_{2}'(t)-\frac{b_{1}(t) c_{3}(t)}{\varepsilon } \right) s + b_{2}(t)\right)
	 +
	\sin\left(\phi\right)
	\left(\left(c_{3}'(t)+\frac{b_{1}(t) c_{2}(t)}{\varepsilon} \right) s  + b_{3}(t)\right)
	\\\\
	-\sin \left(\phi\right)
	\left(\left( c_{2}'(t)-\frac{b_{1}(t) c_{3}(t)}{\varepsilon } \right)s\, +\, b_{2}(t)\right) + 
	\cos \left(\phi\right)
	\left(\left(c_{3}'(t) + 
	\frac{b_{1}(t) c_{2}(t)}{\varepsilon} \right)s + b_{3}(t)\right)
\end{gathered}
\right) , $$
\\\\
where  $c_{1}(t)\,,\,c_{2}(t)\,,\,c_{3}(t)\,,\,b_{1}(t)\,,\,b_{2}(t)\,,\,b_{3}(t) $ are arbitrary smooth functions, $
\phi\, =\, \frac{1}{\varepsilon }\,s\,  c_{1}(t)  $ and the column matrices are relative to the basis $\{u_{1}\,,\,u_{2}\,,\,u_{3}\} $.
\end{theorem}
\hfil

\begin{proof}
	In order to obtain $F_{s}$ we just have to solve the geodesic equation $\nabla_{F_{s}}F_{s}\,=\,0 $ . Then, by expanding it in the frame $\{u_{1}\,,\,u_{2}\,,\,u_{3}\} $ , we obtain the system of equations:
	\\
	$$	\left\{
	\begin{aligned}
	    &\pd{F_{s}^{1}}{s}  \ =\ 0
	    \\\\
	    &\pd{F_{s}^{2}}{s}\ -\ 
	    \frac{1}{\varepsilon }F_{s}^{1} F_{s}^{3} \ =\ 0
	    \\\\
		&\pd{F_{s}^{3}}{s}\ +\ 
		\frac{1}{\varepsilon } F_{s}^{1} F_{s}^{2} \ =\ 0
	\end{aligned}\right. $$
	\\[-5pt]
		
	By integration the first equation, we obtain $F_{s}^{1}\,=\,c_{1}(t) $, then, substituting it in the last two equations, we get a linear system of differential equations whose general solution is
	$$F_{s}^{2} = c_{2}(t) \cos \left(\frac{s}{\varepsilon }  c_{1}(t)\right)\ +\ 
			c_{3}(t) \sin \left(\frac{s}{\varepsilon }  c_{1}(t)\right) \ \text{ and }\ 
	F_{s}^{3} = -c_{2}(t) \sin \left(\frac{s}{\varepsilon }  c_{1}(t)\right)\ +\ 
			c_{3}(t) \cos \left(\frac{s}{\varepsilon }  c_{1}(t)\right) . $$
	\\[-10pt]
	
	As explained in Section \ref{Sec 3}, in order to obtain $F_{t}$ we have to solve the equation given by the integrability condition $\nabla_{F_{t}}F_{s}\,-\,\nabla_{F_{s}}F_{t}\ =\ 0$. Then, expanding it in the frame $\{u_{1}\,,\,u_{2}\,,\,u_{3}\} $ , we obtain the system:
	
	$$\left\{\begin{aligned}
		&\pd{F_{s}^{1}}{t}\ -\ \pd{F_{t}^{1}}{s} \ =\ 0
		\\\\
		&\pd{F_{s}^{2}}{t}\ -\ \pd{F_{t}^{2}}{s} \ -\ 
		\frac{1}{\varepsilon } F_{t}^{1} F_{s}^{3}\ +\ 
		\frac{1}{\varepsilon } F_{t}^{3} F_{s}^{1} \ =\ 0
		\\\\
		&\pd{F_{s}^{3}}{t}\ -\ \pd{F_{t}^{3}}{s} \ +\ 
		\frac{1}{\varepsilon } F_{t}^{1} F_{s}^{2} \ -\ 
		\frac{1}{\varepsilon } F_{t}^{2} F_{s}^{1} \ =\ 0
	\end{aligned}\right. $$
	\\
	
	Substituting the values of $F_{s}^{1}\,,\,F_{s}^{2}\,,\,F_{s}^{3}$ and solving the resulting system we obtain
	$$\left\{\begin{aligned}
		&F_{t}^{1} = b_{1}(t)+s c_{1}'(t) 
		\\\\
		&F_{t}^{2} = \cos\left(\frac{c_{1}(t)}{\varepsilon } s \right) 
		\left(\left(c_{2}'(t)-\frac{1}{\varepsilon } b_{1}(t) c_{3}(t)\right) s + b_{2}(t)\right)
		  +  
		\sin\left(\frac{c_{1}(t)}{\varepsilon} s \right) 
		\left(\left(c_{3}'(t)+\frac{1}{\varepsilon} b_{1}(t) c_{2}(t)\right) s  + b_{3}(t)\right)
		\\\\
		&F_{t}^{3} =  - sin \left(\frac{c_{1}(t)}{\varepsilon } s \right) 
		\left(\left(c_{2}'(t) - \frac{1}{\varepsilon } b_{1}(t) c_{3}(t)\right)s  +\, b_{2}(t)\right) + 
		\cos \left(\frac{c_{1}(t)}{\varepsilon} s \right) 
		\left(\left(c_{3}'(t) +
		\frac{1}{\varepsilon} b_{1}(t) c_{2}(t)\right)s + b_{3}(t)\right)
	\end{aligned}\right.  $$
\end{proof}
\hfill
\\[-10pt]

\subsection{The parametrized ruled surfaces for $p=0$}
As described in Section \ref{Sec 3}, we obtained the tangent frames $\{F_t, F_s\}$ using only the Lie algebra of $\Sol(3)$, without referencing its product structure. And, if one desires, this is enough to also characterize the minimal ruled surfaces, as computing the mean curvature requires only the tangent vectors. However, our goal here is to also obtain the explicit parametrization of these surfaces in the canonical global chart. For that we  first expand our orthonormal left invariant frame in the canonical global chart of $\Sol(3) $, obtaining
$$u_{1} \ 
=\ \frac{1}{\varepsilon}\pdo{x}\,,\  u_{2}\ 
=\ \frac{1}{\alpha}\,\cos(x)\,\pdo{y}\ +\ \frac{1}{\alpha}\,\sin(x)\,\pdo{z} 
\, , \ 
u_{3}\ 
=\ -\frac{1}{\alpha}\,\sin(x)\,\pdo{y}\ +\ \frac{1}{\alpha}\,\cos(x)\,\pdo{z} $$
\\[-10pt]

Then, expanding the formula for $F_{s}$ from Theorem \ref{Th1 Sol3 espec} in terms of this canonical basis yields
$$F_{s}\ 
=\ -\frac{c_{1}}{\varepsilon}\,\pdo{x}
\ +\ 
\left( 
\frac{c_{2}}{\alpha}\,
\cos\left(\phi - x(t ,s)\right)  + 
\frac{c_{3}}{\alpha}\,
\sin\left(\phi - x(t , s)\right)
\right)\,\pdo{y}\ -\ \left( 
\frac{c_{3}}{\alpha}\,
\cos\left(\phi - x(t , s)\right)  - 
\frac{c_{2}}{\alpha}\,
\sin\left(\phi - x(t\,,\,s)\right)
\right)\,\pdo{z} $$
and, since 
$$F_{s}\ 
=\ \pd{x}{s}\,\pdo{x}\ +\ \pd{y}{s}\,\pdo{y} \ +\ \pd{z}{s}\,\pdo{z} , $$
\\
then, by solving the system of equations, one obtains the general solution:

$$F(t\,,\,s)\ 
=\ \left(A_{1} \ +\ \frac{c_{1}}{\varepsilon}\,s \ ,\ 
	A_{2} + 
	\frac{1}{\alpha}\left(c_{2}\,
	\cos(A_{1})  - c_{3}\,\sin(A_{1})\right)\,s \ ,\ 
	A_{3} + 
	\frac{1}{\alpha}\left(c_{3}\,
	\cos(A_{1}) \ +\ c_{2}\,\sin(A_{1})\right)\,s
\right) $$
\\[-10pt]

Where $A_{1}(t)\,,\,A_{2}(t)\,,\,A_{3}(t) $ are arbitrary smooth functions that come from the integration of the equations. In order to relate the $A_{i}$'s and $b_{i}$'s we just have to expand $F_{t}$ obtained above in the basis $\{u_{1}\,,\,u_{2}\,,\,u_{3}\}$ and subtract from the expression for $F_{t}$ found in Theorem(\ref{Th1 Sol3 espec}) and identify when this difference is null. Doing that we get

$$A_{1}' \ =\ -\frac{1}{\varepsilon}\,b_{1} \ ,\ 
A_{2}' \ 
	=\ \frac{1}{\alpha}\,\left(
	b_{2}\,\cos(A_{1}) \ -\ b_{3}\,\sin(A_{1})\right) \ ,\ 
	A_{3}'\ 
		=\ \frac{1}{\alpha}\,\left(
		b_{3}\,\cos(A_{1}) \ +\ b_{2}\,\,\sin(A_{1})\right) . $$
\\[-10pt]

We summarize the results obtained in the following theorem:
\begin{theorem}\label{Ruled Surf Param Espec}
    Let $F:(t_{1}\,,\,t_{2})\times (s_{1}\,,\,s_{2})\to \Sol(3)$ be a ruled surface, then it must have the following form:
    \\\\
    $$F(t\,,\,s)\ 
    =\ \left(A_{1} \ +\ \frac{c_{1}}{\varepsilon}\,s \ ,\ 
    	A_{2} + 
    	\frac{1}{\alpha}\left(c_{2}\,
    	\cos(A_{1})  - c_{3}\,\sin(A_{1})\right)\,s \ ,\ 
    	A_{3} + 
    	\frac{1}{\alpha}\left(c_{3}\,
    	\cos(A_{1}) \ +\ c_{2}\,\sin(A_{1})\right)\,s
    \right)   , $$
\hfil
\\
\vspace{3pt}
where $c_{1}(t)\,,\,c_{2}(t)\,,\,c_{3}(t)\,,\,A_{1}(t)\,,\,A_{2}(t)\,,\,A_{3}(t) $ are arbitrary smooth functions.

\end{theorem}


\subsection{Analysis of the Mean Curvature}
Since now we are working with general ruled surfaces the formula for the mean curvature is 
\begin{equation}\label{Min Cond Sol3 espec}	H \ 
=\  \frac{1}{\|F_{t}\wedge F_{s}\|_{g}^{3}}\,
g\left(\|F_{s}\|_{g}^{2}\,\nabla_{F_{t}}F_{t}  \ -\ g(F_{t}\,,\,F_{s})\,\left(\nabla_{F_{t}}F_{s} \ +\ \nabla_{F_{s}}F_{t}\right) ,\ F_{t}\times F_{s}\right) .
\end{equation}
\\
Also, just as it happens in the case where $p\neq 0$, it is necessary to simplify the expression for $ \|F_{t}\times F_{s}\|^{3}_{g}\,H$ in order to solve it, and, for that we take $b_{1}$ identically null. Having done that we obtain the following expression for the mean curvature

$$\begin{aligned}\label{Min Cond Sol3 especb1}
	&&\|F_{t}\times F_{s}\|^{3}_{g}\,H  =
	&-\|\vec{c}\|^{2}\,
	\left|\begin{matrix}
	c_{1} & c_{2} & c_{3} \\
	c_{1}' & c_{2}' & c_{3}' \\
	c_{1}'' & c_{2}'' & c_{3}'' \\
	\end{matrix}\right|\,s^{2}  +  
	\left(\inn{\left(\|\vec{c}\|^{2}\vec{b}\right)' ; \vec{c}'\times \vec{c}} - 
	\inn{\|\vec{c}\|^{2}\vec{b} ; \left(\vec{c}'\times \vec{c}\right)'}\right)s\, +\,\\\\[-5pt]
	&&&-\, c_{1}\,
	\left|\begin{matrix}
		b_{2} & b_{3} \\
		b_{2}' & b_{3}'
	\end{matrix}\right|\,\|\vec{c}\|^{2}  +  2\,\inn{\vec{b} , \vec{c}}\,\inn{\vec{b} , \vec{c}'\times \vec{c}} ,
\end{aligned} $$
\\
where $\vec{b}(t)\ =\ (0\,,\,b_{2}(t)\,,\,b_{3}(t))\,$ and $\,\vec{c}(t) \ =\ (c_{1}(t)\,,\,c_{2}(t)\,,\,c_{3}(t)) $.
\\

We now focus on the coefficient of $s^{2}$. Given that $\|F_{s}\|_{g}^{2} = \|\vec{c}\|^{2}$ and $F_{s}$ cannot be identically zero, it follows that
$$\left|\begin{matrix}
c_{1} & c_{2} & c_{3} \\
c_{1}' & c_{2}' & c_{3}' \\
c_{1}'' & c_{2}'' & c_{3}'' \\
\end{matrix}\right| \ =\ 0  . $$
\\[-10pt]

Since a determinant is zero if its columns or lines are linearly dependent we separate our analysis in cases.
\\\\
\textbf{(Case \RomanNum{1})} Given the frequency of the term $\vec{c}'\times \vec{c} $ in the expression we choose $\vec{c}'  = f(t)\,\vec{c}$, where $f\in C^{\infty}((t_{1}\,,\,t_{2})) $  is an arbitrary function, then $$\vec{c}(t) \ =\ \vec{c}_{0}\,exp\left(\int_{t_{0}}^{t}f(r)\,dr\right) \ \ \text{ , where }\ \ 
\vec{c}_{0}\,=\,(c_{01}\,,\,c_{02}\,,\,c_{03}) \in \mathbb{R}^{3} . $$
\\[-10pt]

With such choice we  end up with 
$$\|F_{t}\times F_{s}\|^{3}_{g}\,H \ 
=\ -\ c_{1}\,
\left|\begin{matrix}
	b_{2} & b_{3} \\
	b_{2}' & b_{3}'
\end{matrix}\right|\,\|\vec{c}\|^{2}  \ 
=\ 
-\ c_{10}\,
\left|\begin{matrix}
	b_{2} & b_{3} \\
	b_{2}' & b_{3}'
\end{matrix}\right|\,\left(c_{10}^{2}\ +\ c_{20}^{2}\ +\ c_{30}^{2}\right)\,
exp\left(3\,\int_{t_{0}}^{t}f(r)\,dr\right) , $$
which vanishes if $c_{10}\,=\,0 $, or $ b_{2}\,=\,0 $, or there exists a $k\in \mathbb{R}$ such that $b_{3}\,=\,k\,b_{2} $. So we explore these three sub-cases separately.
\\\\
\textbf{(\RomanNum{1}.1)} If $c_{10}=0 $, then the expression for the parametrized ruled surface of Theorem \ref{Ruled Surf Param Espec} becomes
\\
$$F(t\,,\,s)\ =\ 
\left(0\ ,\ 
\frac{B_{2}(t)}{\alpha}\ +\ \frac{c_{02}}{\alpha}\,e^{G(t)}\,s \ ,\ 
\frac{B_{3}(t)}{\alpha}\ +\ \frac{c_{03}}{\alpha}\,e^{G(t)}\,s\right) ,
$$
\\
$$\text{where } B_{2}(t)=\int_{t_{0}}^{t}b_{2}(r)\,dr \ ,\ 
B_{3}(t)=\int_{t_{0}}^{t}b_{3}(r)\,dr \ \text{ and }\ 
G(t)\ =\ \int_{t_{0}}^{t}f(r)\,dr . $$
\\
Also, since $\|F_{t}\times F_{s}\|_{g} \ =\ \exp(G)\,\left(b_{2}\,c_{03}\ -\ b_{3}\,c_{02}\right)$, we require this determinant to be non-vanishing to ensure that the surface $F$ does not degenerate into a curve.
\\\\
\textbf{(\RomanNum{1}.2)} If $b_{2}$ is identically null, then the parametrization of the ruled surfaces becomes
$$F(t\,,\,s)\ =\ 
	\left(
		\frac{c_{01}}{\varepsilon}e^{G(t)}\,s \ ,\ 
		\frac{c_{02}}{\alpha}\,e^{G(t)}\,s \ ,\ 
		\frac{c_{03}}{\alpha}\,e^{G(t)}\,s \ +\ \frac{1}{\alpha}\,B_{3}(t)
	\right)  $$

In this situation, $\|F_{t}\times F_{s}\|^{2}\ 
=\ b_{3}^{2}\,exp\left(2\,G(t)\right)\,
\left(c_{01}^{2}\,+\,c_{02}^{2}\right)  , $ meaning that we have to impose that $b_{3}$ is non vanishing and $c_{01}^{2}\,+\,c_{02}^{2}\neq 0 $. 
\\\\
\textbf{(\RomanNum{1}.3)} If $b_{3}\,=\,k\,b_{2}$, then the ruled surface will be 
$$F(t\,,\,s)\ =\ 
F(t,s)\ 
=\ \left(\frac{c_{01}}{\varepsilon}\,e^{G(t)}\,s \ ,\ 
\frac{c_{02}}{\alpha}\,e^{G(t)}\,s \ +\ \frac{1}{\alpha}\,B_{2}(t) \ ,\ 
\frac{c_{03}}{\alpha}\,e^{G(t)}\,s \ +\ \frac{k}{\alpha}\,B_{2}(t)\right) , $$
\\
Also, in this situation, 
$$\|F_{t}\times F_{s}\|^{2}\ 
=\ b_{2}^{2}\,\left(\left(1\,+\, k^{2}\right)\,c_{01}^{2}\ +\ 
\left(k\,c_{02}\,-\,c_{03}\right)^{2}\right)\,
exp\left(2\,G(t)\right),  $$
\\
which is non vanishing if $b_{2} $ also is a nowhere vanishing function and $c_{10}\neq 0 $ or $k\,c_{2}\neq c_{3} $.
\\\\
\textbf{(Case \RomanNum{2})} We now suppose that $\vec{b} = 0$ and that there exists $k_{1}\,,\,k_{2}\in \mathbb{R} $ such that $c_{1} = k_{1}\,c_{2} + k_{2}\,c_{3}$, then the expression for the parametrized ruled surface of Theorem \ref{Ruled Surf Param Espec} becomes
$$F(t\,,\,s)\ =\ 
\left(\frac{1}{\varepsilon}\,\left(k_{1}\,c_{2}\ +\ k_{2}\,c_{3}\right)\,s \ ,\ 
\frac{c_{2}(t)}{\alpha}\,s\ ,\ 
\frac{c_{3}(t)}{\alpha}\,s
\right) $$

Also, since $\|F_{t}\times F_{s}\|^{2}_{g} \ 
=\ \left(1 + k_{1}^{2} + k_{2}^{2}\right)\,
\left(c_{3}c_{2}'\, -\, c_{2}c_{3}'\right)^{2}\,s^{2} , $ we must require that the following determinant does not vanish
$$\left|\begin{matrix}
	c_{2} & c_{3} \\
	c_{2}' & c_{3}'
\end{matrix}\right| ,$$
\\
which is equivalent to say that $ c_{3}$ and $c_{2} $ are never linearly dependent functions.
\\

We now go back to the situation where $b_{1}$ is not identically zero and apply the conditions found in cases (\RomanNum{1}) and (\RomanNum{2}) in order to simplify the equation \eqref{Min Cond Sol3 espec}. Assuming the hypothesis of \textbf{(Case \RomanNum{1}.1)}, being $\vec{c}' \ =\ f(t)\,\vec{c} $ and $c_{01}=0 $, then we have that

$$\|F_{t}\times F_{s}\|^{3}H \ 
=\ \left(c_{2}^{2}\,+\,c_{3}^{2}\right)\,
\left(
\left|\begin{matrix}
	b_{2} & c_{2}\\
	b_{3} & c_{3}
\end{matrix}\right|\,b_{1}' \ +\ 
\frac{q}{\varepsilon}\,\left(b_{2}\,c_{2}\, +\, b_{3}\,c_{3}\right)\,b_{1}^{2} \ -\ 
\left|\begin{matrix}
	b_{2}' & c_{2}\\
	b_{3}' & c_{3}
\end{matrix}\right|\,b_{1}
\right) $$
\\
Supposing the function given by the determinant that multiplies $b_{1}' $ is non vanishing, we can see this expression as a Bernoulli differential equation for $b_{1}(t)$ and, by solving it we obtain
$$b_{1}(t)\ 
=\ \exp\left(L(t)\right)
\left(k  + \frac{q}{\varepsilon}\int_{t_{0}}^{t}\exp(L(r))
\frac{b_{2}(r)\,c_{02} + c_{03}\,b_{3}(r)}{b_{2}(r)\,c_{03} -\,b_{3}(r)\,c_{02}}\, dr\right)^{-1}
\text{ , where }\ L(t)\ =\ \int_{t_{0}}^{t}\frac{c_{03}\,b_{2}'(r)\,-\,c_{02}\,b_{3}'(r)}{b_{2}(r)\,c_{03}\,-\,b_{3}(r)\,c_{02}}\ dr . $$
\\
Substituting the restrictions above into the parametrization from Theorem \ref{Ruled Surf Param Espec}, we obtain the surface
\\
$$F(t,s)\ 
=\ \left(\begin{gathered}
	-\frac{1}{\varepsilon}\,B_{1}(t) \\\\
	\frac{1}{\alpha}\,\int_{t_{0}}^{t} b_{2}\,\cos\left(\frac{1}{\varepsilon}B_{1}\right) + b_{3}\,\sin\left(\frac{1}{\varepsilon}\,B_{1}\right)\,dt' \ +\ 
	\frac{1}{\alpha}\exp\left(G(t)\right)
	\left(c_{02}\,\cos\left(\frac{1}{\varepsilon}B_{1}\right) +
	c_{03}\,\sin\left(\frac{1}{\varepsilon}\,B_{1}\right)\right)\,s \\\\
	\frac{1}{\alpha}\,\int_{t_{0}}^{t} b_{3}\,\cos\left(\frac{1}{\varepsilon}B_{1}\right) - b_{2}\,\sin\left(\frac{1}{\varepsilon}\,B_{1}\right)\,dt' \ +\ 
	\frac{1}{\alpha}\exp\left(G(t)\right)
	\left(c_{03}\,\cos\left(\frac{1}{\varepsilon}B_{1}\right) - 
	c_{02}\,\sin\left(\frac{1}{\varepsilon}\,B_{1}\right)\right)\,s
\end{gathered}\right), $$
\\
where $$B_{1}(t)\ =\ \int_{t_{0}}^{t}b_{1}(r)\,dr . $$

Then, in this situation,
$$\|F_{t}\times F_{s}\|^{2}\ 
=\ b_{1}^{2}\,\left(c_{2}^{2}\,+\,c_{3}^{2}\right) \ +\ 
\left(\frac{q}{\varepsilon}\,b_{1}\,\left(c_{2}^{2}\,+\,c_{3}^{2}\right)\,s \ +\ 
\left|\begin{matrix}
	b_{2} & c_{2}\\
	b_{3} & c_{3}
\end{matrix}\right|
\right)^{2} , $$
which, by our assumption is nowhere vanishing if $b_{1}$ also is a non vanishing function and $c_{01}^{2}+c_{03}^{2}\neq 0 $.
\\

Now, turning our attention to the case where the determinant is identically null, we will have that $b_{2}= l\,c_{2} $ and $b_{3}= l\,c_{3} $ for some $l\in \mathbb{R} $. Then, the expression for the mean curvature becomes
$$\|F_{t}\times F_{s}\|^{3}_{g}\,H\ 
=\ -\frac{l}{\varepsilon}\,b_{1}\,\left(c_{02}^{2}\ +\ c_{03}^{2}\right)\,\exp\left(2G(t)\right) . $$
\\
Since $b_{1}$ is taken as not identically null and $c_{02} $, $c_{03} $ can't be both zero, because otherwise we have $\|F_{s}\|=0 $, our only alternative is to take $l=0 $. Substituting these restrictions into the parametrization from Theorem \ref{Ruled Surf Param Espec} we obtain the surface
\\
$$F(t,s)\ 
=\ \left(
	-\frac{B_{1}}{\varepsilon} \ ,\
	\frac{e^{G(t)}}{\alpha}
	\left(c_{02}\,\cos\left(\frac{B_{1}}{\varepsilon}\right) + 
	c_{03}\,\sin\left(\frac{B_{1}}{\varepsilon}\right)\right)\,s \ ,\
	\frac{e^{G(t)}}{\alpha}
	\left(c_{03}\,\cos\left(\frac{B_{1}}{\varepsilon}\right) + 
	c_{02}\,\sin\left(\frac{B_{1}}{\varepsilon}\right)\right)\,s
\right) . $$
\\

Assuming conditions of \textbf{(Case \RomanNum{1}.2)}, meaning $\vec{c}'= f\,\vec{c} $ and $b_{2} $ identically null, equation \eqref{Min Cond Sol3 espec} becomes
$$\|F_{t}\times F_{s}\|^{3}_{g}\,H \ 
=\ \left(\frac{q}{\varepsilon}\right)^{3}\,b_{1}^{3}\,c_{1}\,
\left(c_{1}^{2} + c_{2}^{2}\right)\,\left(c_{1}^{2} + c_{2}^{2} + c_{3}^{2}\right)\,s \ +\ (\dots)\,s\ +\ (\dots), $$
\\
then, in order to not repeat the last case, we have to impose that $c_{02}=c_{03}=0 $. Doing that the expression simplifies to 
$$\|F_{t}\times F_{s}\|^{3}_{g}\,H \ 
=\ \frac{q}{\varepsilon}\,b_{1}\,b_{3}^{2}\,c_{1}^{2}, $$
meaning that we have to also impose that $b_{3} $ vanishes everywhere. However, since $\|F_{t}\times F_{s}\|^{2}_{g}\ 
=\ b_{3}^{2}\,c_{1}^{2} $, this can't happen. Therefore, this case does not yield a minimal surface.
\\

Assuming conditions of \textbf{(Case \RomanNum{1}.3)}, meaning $\vec{c}'= f\,\vec{c} $ and $b_{3} = k\,b_{2} $ for some $k\in \mathbb{R} $, the analysis is completely analogous to the situation before and also does not yield a minimal surface.
\\\\

Assuming conditions of \textbf{(Case \RomanNum{2})}, meaning $b_{2} $ and $b_{3} $ identically null and $c_{1} = k_{1}\,c_{2} + k_{2}\,c_{3}$ for some $k_{1}\,,\,k_{2}\in \mathbb{R} $, then 
$$\|F_{t}\times F_{s}\|^{3}_{g}\,H \ 
=\ (\dots)\,s^{2}\ +\ (\dots)\,s \ +\ 
2\,b_{1}^{2}\,\left(k_{1}\,c_{2} + k_{2}\,c_{3}\right)\,
\left(\frac{q}{\varepsilon}\,b_{1}\,\left(c_{2}^{2} + c_{3}^{2}\right) \ +\ 
\left(c_{3}\,c_{2}' - c_{2}\,c_{3}'\right)\right). $$
\\[-5pt]

However, by simply taking 
$$b_{1}\ 
=\ \varepsilon\,
\frac{c_{3}\,c_{2}' - c_{2}\,c_{3}'}{c_{2}^{2} + c_{3}^{2}} $$
the entire expression for $\|F_{t}\times F_{s}\|^{3}_{g}\,H $ vanishes, achieving our goal. Substituting these restrictions into the parametrization from Theorem \ref{Ruled Surf Param Espec} we obtain the surface
$$F(t,s)\ 
=\ \left(\tan^{-1}\left(\frac{c_{3}}{c_{2}}\right)  + \frac{s}{\varepsilon}\,(k_{1}c_{2} + k_{2}c_{3}) \ ,\ 
\frac{s}{\alpha}\,
\left(1 + \left(\frac{c_{3}}{c_{2}}\right)^{2}\right)^{-\frac{1}{2}}
\left(c_{2} - \frac{c_{3}^{2}}{c_{2}}\right) \ ,\ 
\frac{2\,s}{\alpha}\,
\left(1 + \left(\frac{c_{3}}{c_{2}}\right)^{2}\right)^{-\frac{1}{2}} c_{3} 
\right) $$
\\
Also, we have that
$$\|F_{t}\times F_{s}\|^{2}_{g} \ 
=\ \left(\frac{\varepsilon}{q}  + \left(k_{2}\,c_{2} - k_{1}\,c_{3}\right)\,s\right)^{2}\,
\frac{(c_{3}c_{2}' - c_{2}c_{3}')^{2}}{c_{2}^{2} + c_{3}^{2}} , $$
\\
meaning that we have to require that $c_{3}c_{2}' - c_{2}c_{3}' $ is a non vanishing function.

\hfil
\\\\
We summarize the results proven above in the following theorem:
\\
\begin{theorem}\label{The Min Surf Espec}
	The following are ruled minimal surfaces in $\Sol(3)$:
	\\\\
	(1) $$F(t\,,\,s)\ =\ 
	\left(0\ ,\ 
	\frac{B_{2}(t)}{\alpha}\ +\ \frac{c_{02}}{\alpha}\,e^{G(t)}\,s \ ,\ 
	\frac{B_{3}(t)}{\alpha}\ +\ \frac{c_{03}}{\alpha}\,e^{G(t)}\,s\right) , $$
	\\
	$\text{where } c_{02}\,,\,c_{03}\in \mathbb{R},\ G(t)\text{ is an arbitrary smooth function and } $ $B_{2}(t)\,,\,B_{3}(t) $ are smooth functions such that, for all values of $t\in (t_{1},t_{2})$, $c_{03}\,B'_{2}(t)\neq c_{02}\,B_{3}'(t) $.
	\\\\
	(2) $$F(t\,,\,s)\ =\ 
		\left(
			\frac{c_{01}}{\varepsilon}\,e^{G(t)}\,s \ ,\ 
			\frac{c_{02}}{\alpha}\,e^{G(t)}\,s \ ,\ 
			\frac{c_{03}}{\alpha}\,e^{G(t)}\,s \ +\ \frac{1}{\alpha}\,B_{3}(t)
		\right) , $$
		\\
		$\text{where } G(t) \text{ is an arbitrary smooth function, } B_{3}(t)$ is a smooth function whose derivative never vanishes and $c_{01}\,,\,c_{02}\,,\,c_{03}\in \mathbb{R}$ such that $c_{01}^{2} + c_{02}^{2}\neq 0 $.
		\\\\ 
	(3)$$F(t,s)\ 
		=\ \left(\frac{c_{01}}{\varepsilon}\,e^{G(t)}\,s \ ,\ 
		\frac{c_{02}}{\alpha}\,e^{G(t)}\,s \ +\ \frac{1}{\alpha}\,B_{2}(t) \ ,\ 
		\frac{c_{03}}{\alpha}\,e^{G(t)}\,s \ +\ \frac{k}{\alpha}\,B_{2}(t)\right) , $$
		\\
		$\text{where } G(t) \text{ is an arbitrary smooth function, } B_{2}(t)  $ is a smooth function whose derivative never vanishes and $k\,,\,c_{01}\,,\,c_{02}\,,\,c_{03}\in \mathbb{R} $ such that $c_{01}\neq 0 \text{ or } c_{03}\neq k\,c_{02} $.
		\\\\
	(4) $$F(t\,,\,s)\ =\ 
	\left(\frac{1}{\varepsilon}\,\left(k_{1}\,c_{2}\ +\ k_{2}\,c_{3}\right)\,s \ ,\ 
	\frac{c_{2}(t)}{\alpha}\,s\ ,\ 
	\frac{c_{3}(t)}{\alpha}\,s
	\right), 
	\text{where } c_{2}(t)\,,\,c_{3}(t) \text{ are smooth functions such that for all  }  $$
	\\
	values of $t\in (t_{1},t_{2})$, $c_{2}(t)c_{3}'(t) \neq c_{3}(t)c_{2}'(t) $ and $k_{1}\,,\,k_{2}\in \mathbb{R}. $
	\\\\
	(5) $$F(t,s)\ 
	=\ \left(\begin{gathered}
		-\frac{1}{\varepsilon}\,B_{1} \\\\
		\frac{1}{\alpha}\,\int_{t_{0}}^{t} b_{2}\,\cos\left(\frac{1}{\varepsilon}B_{1}\right) + b_{3}\,\sin\left(\frac{1}{\varepsilon}\,B_{1}\right)\,dt' \ +\ 
		\frac{1}{\alpha}\exp\left(G(t)\right)
		\left(c_{02}\,\cos\left(\frac{1}{\varepsilon}B_{1}\right) +
		c_{03}\,\sin\left(\frac{1}{\varepsilon}\,B_{1}\right)\right)\,s \\\\
		\frac{1}{\alpha}\,\int_{t_{0}}^{t} b_{3}\,\cos\left(\frac{1}{\varepsilon}B_{1}\right) - b_{2}\,\sin\left(\frac{1}{\varepsilon}\,B_{1}\right)\,dt' \ +\ 
		\frac{1}{\alpha}\exp\left(G(t)\right)
		\left(c_{03}\,\cos\left(\frac{1}{\varepsilon}B_{1}\right) - 
		c_{02}\,\sin\left(\frac{1}{\varepsilon}\,B_{1}\right)\right)\,s
	\end{gathered}\right), $$
	\\
	where $G(t)$ is an arbitrary smooth function, $c_{02}\,,\,c_{03}\in \mathbb{R}$, $b_{2}(t)\,,\,b_{3}(t) $ are smooth functions such that $c_{03}b_{2}(t)\neq c_{02}b_{3}(t) $ for all values of $t\in (t_{1},t_{2}) $ and 
	$$B_{1}(t)\ 
	=\ \int_{t_{0}}^{t}\exp\left(L(t')\right)
	\left(k  + \frac{q}{\varepsilon}\int_{t_{0}}^{t'}\exp(L(r))
	\frac{b_{2}(r)\,c_{02} + c_{03}\,b_{3}(r)}{b_{2}(r)\,c_{03} -\,b_{3}(r)\,c_{02}}\, dr\right)^{-1} dt'
	\text{ , where } L(t) = \int_{t_{0}}^{t}\frac{c_{03}\,b_{2}'(r)\,-\,c_{02}\,b_{3}'(r)}{b_{2}(r)\,c_{03}\,-\,b_{3}(r)\,c_{02}} dr . $$
	\\
	(6) $$F(t,s)\ 
	=\ \left(
		-\frac{B_{1}}{\varepsilon} \ ,\
		\frac{e^{G(t)}}{\alpha}
		\left(c_{02}\,\cos\left(\frac{B_{1}}{\varepsilon}\right) + 
		c_{03}\,\sin\left(\frac{B_{1}}{\varepsilon}\right)\right)\,s \ ,\
		\frac{e^{G(t)}}{\alpha}
		\left(c_{03}\,\cos\left(\frac{B_{1}}{\varepsilon}\right) + 
		c_{02}\,\sin\left(\frac{B_{1}}{\varepsilon}\right)\right)\,s
	\right) ,  $$
	\\
	$\text{where }G(t) \text{ is an arbitrary smooth function  }, $ $B_{1}(t) $ is a smooth function whose derivative does not vanishes and $c_{02}\,,\,c_{03}\in \mathbb{R} $.
	\\\\
	(7) $$F(t,s)\ 
	=\ \left(\tan^{-1}\left(\frac{c_{3}}{c_{2}}\right)  + \frac{s}{\varepsilon}\,(k_{1}c_{2} + k_{2}c_{3}) \ ,\ 
	\frac{s}{\alpha}\,
	\left(1 + \left(\frac{c_{3}}{c_{2}}\right)^{2}\right)^{-\frac{1}{2}}
	\left(c_{2} - \frac{c_{3}^{2}}{c_{2}}\right) \ ,\ 
	\frac{2\,s}{\alpha}\,
	\left(1 + \left(\frac{c_{3}}{c_{2}}\right)^{2}\right)^{-\frac{1}{2}} c_{3} 
	\right),  $$
	where $k_{1}\,,\,k_{2}\in \mathbb{R} $ and $c_{2}$ and $c_{3}c_{2}' - c_{2}c_{3}'  $ are non vanishing functions.
\end{theorem}

\section*{Acknowledgment}
I sincerely thank my advisor Lino Grama  for his invaluable guidance throughout this research. I would also to thank Velichka Milousheva and Aleksandar Petkov for helpful comments. This research was supported by a grant from the "Coordenação de Aperfeiçoamento de Pessoal de Nível Superior" (CAPES).


\end{document}